\documentclass{article}
\usepackage[letterpaper, margin=1in]{geometry}

\title{Vector and Matrix Optimal Mass Transport: \\Theory, Algorithm, and Applications} 
\author{Ernest K. Ryu, Yongxin Chen, Wuchen Li, and Stanley Osher}
\date{December 29, 2017}

\usepackage{amsmath,amssymb,euscript,yfonts,psfrag,graphicx}
\usepackage{subfig}
\usepackage{tikz}
\usetikzlibrary{arrows.meta}
\usepackage{accents}
\DeclareMathOperator*{\esssup}{ess\,sup}

\usepackage{amsthm}
\graphicspath{{./},{./figures/}}

\newtheorem{theorem}{Theorem}
\newtheorem{lemma}{Lemma}

\newcommand{\divg}{\mathrm{div}}
\newcommand{\vx}{\mathbf{x}}
\newcommand{\vy}{{\mathbf y}}
\newcommand{\vu}{{\mathbf u}}

\newcommand{\vU}{{\mathbf U}}
\newcommand{\vW}{{\mathbf W}}
\newcommand{\vL}{{\mathbf L}}

\DeclareRobustCommand{\vect}[1]{\accentset{\rightharpoonup}{#1}}

\newcommand{\1}{\>}
\newcommand{\2}{\>\>}
\newcommand{\3}{\>\>\>}
\newcommand{\4}{\>\>\>\>}

\newcommand{\For}{\textbf{for }}

\newcommand{\End}{\textbf{end}}

\DeclareMathOperator*{\argmin}{argmin}
\DeclareMathOperator*{\argmax}{argmax}

\newcommand{\rd}{\mathbb{R}^d}
\newcommand{\rk}{\mathbb{R}^k}
\newcommand{\rkd}{\mathbb{R}^{k\times d}}

\newcommand{\srk}{\mathrm{shrink}}
\newcommand{\mR}{{\mathbb R}}

\newcommand{\cC}{{\mathcal C}}

\newcommand{\cG}{{\mathcal G}}
\newcommand{\cH}{{\mathcal H}}

\newcommand{\cS}{{\mathcal S}}

\newcommand{\diag}{\operatorname{diag}}

\newcommand{\trace}{\operatorname{tr}}

\definecolor{grey}{rgb}{0.6,0.6,0.6}
\definecolor{lightgray}{rgb}{0.97,.99,0.99}

\begin{document}

\maketitle

\begin{abstract}
In many applications such as color image processing, data has more than one piece of information associated with each spatial coordinate, and in such cases the classical optimal mass transport (OMT) must be generalized to handle vector-valued or matrix-valued densities. In this paper, we discuss the vector and matrix optimal mass transport and present three contributions.
We first present a rigorous mathematical formulation for these setups and provide analytical results including existence of solutions and strong duality.
Next, we present a simple, scalable, and parallelizable methods to solve the vector and matrix-OMT problems. Finally, we implement the proposed methods on a CUDA GPU and present experiments and applications.
\end{abstract}

\maketitle

\section{Introduction}
Optimal mass transport (OMT) is a subject with a long history.
Started by Monge \cite{Mon81} and developed by many great mathematicians \cite{Kan42,Bre91,GanMcc96,Mcc97,JorKinOtt98,BenBre00,OttVil00}, the subject now has 
incredibly rich theory and applications. It has found numerous applications in different areas such as partial differential equations, probability theory, physics, economics, image processing, and control \cite{Eva99,HakTanAng04,TanGeoTan10,MueKarKolTan13,CheGeoPav14e,CheGeoPav15b,Che16}. See also \cite{Rac98,Vil03,AmbGigSav06} and references therein.

However, in many applications such as color image processing, there is more than one piece of information associated with each spatial coordinate,
and such data can be interpreted as vector-valued or matrix-valued densities.
As the classical optimal mass transport works with scalar probability densities,
such applications require a new notion of mass transport.
For this purpose, Chen et al.\ \cite{CheGeoTan16,CheGeoTan17,CheGeoNinTan17, chen2017}
recently developed a framework for vector-valued and matrix-valued optimal mass transport. See also \cite{NinGeoTan13,NinGeo14,NinGeoTan15,FitLauSte16,QPT,VogLel17} for other different frameworks. For vector-valued OMT,
potential applications include color image processing, multi-modality medical imaging, and image processing involving textures. For matrix-valued OMT, we have diffusion tensor imaging,  multivariate spectral analysis, and stress tensor analysis.

Several mathematical aspects of vector and matrix-valued OMT were not addressed in the previous work \cite{CheGeoTan16,CheGeoTan17,CheGeoNinTan17}.
As the first contribution of this paper, 
we present duality and existence results
of the continuous vector and matrix-valued OMT problems
along with rigorous problem formulations.

Although the classical theory of OMT is very rich, only recently has there been much attention to numerical methods to
compute the OMT.
Several recent work proposed algorithms 
to solve the $L^2$ OMT
\cite{AngHakTan03,Cut13,BenFroObe14,HabHor15,BenCarCut15,CheGeoPav15a,yongxin,GenCutPeyBac16}
and the $L^1$ OMT \cite{LiRyuOsh17,L1partial}.
As the second contribution of this paper, 
we present first-order primal-dual methods to solve the vector and matrix-valued OMT problems. The methods simultaneously solve for both the primal and dual solutions (hence a primal-dual method) and are scalable
as they are first-order methods. We also discuss the convergence of the methods.


As the third contribution of this paper, 
we implement the proposed method on a CUDA GPU and present several applications. The proposed algorithms' simple structure
allows us to
 effectively utilize the computing capability of the CUDA architecture, and we demonstrate this through our experiments. We furthermore release the code for scientific reproducibility.

The rest of the paper is structured as follows. In Section \ref{sec:OMT} we give a quick review of the classic OMT theory,
which allows us
 to present the
later sections in an analogous manner and thereby outline the similarities and differences.
In Section \ref{sec:vectorOMT} and Section \ref{sec:matrixOMT}
we present the vector and matrix-valued OMT problems and state a few theoretical results.
In Section \ref{s-duality-proof}, we present and prove the analytical results.
In Section \ref{s:alg-prelim}, we discuss preliminaries we need for 
Section \ref{sec:algorithm}, where we present the algorithm. 
In Section \ref{sec:examples}, we present the experiments and applications.

 \section{Optimal mass transport} \label{sec:OMT}
Let $\Omega\subset \mathbb{R}^d$ be a closed, convex, compact domain.
Let $\lambda^0$ and $\lambda^1$ be nonnegative densities 
supported on $\Omega$ with unit mass, i.e., 
$\int_{\Omega}\lambda^0(\vx)\;d\vx=\int_{\Omega}\lambda^1(\vx)\;d\vx=1$.
Let $\|\cdot\|$ denote any norm on $\mathbb{R}^d$.

In 1781, Monge posed the optimal mass transport (OMT) problem,
which solves
\begin{equation}\label{map}
\begin{split}
\underset{T}{\text{minimize}}\quad \int_{\Omega}\|\vx-T(\vx)\| \lambda^0(\vx)\;d\vx.
\end{split}
\end{equation}
The optimization variable $T:\Omega\rightarrow\Omega$
is smooth, one-to-one, and transfers $\lambda^0(\vx)$ to $\lambda^1(\vx)$.
The optimization problem \eqref{map} is 
nonlinear and nonconvex.
In 1940, Kantorovich
relaxed \eqref{map} into a linear (convex) optimization problem:
\begin{equation}\label{Monge}
S(\lambda^0,\lambda^1)=
\left(
\begin{array}{ll}
\underset{\pi}{\text{minimize}}&
\int_{\Omega\times \Omega}\|\vx-\vy\| \pi(\vx,\vy)\;d\vx d\vy\\
\mbox{subject to} &
\pi(\vx,\vy)\ge 0\\
&\int_{\Omega}\pi(\vx,\vy)\;d\vy=\lambda^0(\vx)\\
&\int_{\Omega}\pi(\vx,\vy)\;d\vx=\lambda^1(\vy).
\end{array}
\right)
\end{equation}
The optimization variable $\pi$ is a joint nonnegative measure on $\Omega\times \Omega$
having  $\lambda^0(\vx)$ and $\lambda^1(\vy)$ as marginals.
To clarify, $S(\lambda^0,\lambda^1)$ denotes the optimal value of \eqref{Monge}.

\subsection{Scalar optimal mass transport}
The theory of optimal transport \cite{GanMcc96,Vil03,Vil08} remarkably points out that \eqref{Monge} is equivalent to the following flux minimization problem:
\begin{equation}\label{Kan1}
S(\lambda^0,\lambda^1)=
\left(
\begin{array}{ll}
\underset{\vu}{\text{minimize}}&
\int_{\Omega}\|\vu(\vx)\| \;d\vx\\
\mbox{subject to} &
\divg_\vx (\vu)(\vx)=\lambda^0(\vx)-\lambda^1(\vx)\\
&\vu(\vx)^T \mathbf{n}(\vx)=0,\,\,
\mbox{for all }\begin{cases}
\vx\in \partial \Omega,\\
\text{$ \mathbf{n}(\vx)$ normal to $\partial\Omega$},
\end{cases}
\end{array}
\right)
\end{equation}
where $\vu=(u_1,\dots,u_d): \Omega\rightarrow \mathbb{R}^d$ is the optimization variable
and $\divg_\vx$ denote the (spatial) divergence operator. Although \eqref{Kan1} and \eqref{Monge} are mathematically equivalent,
\eqref{Kan1} is much more computationally effective as its optimization variable
$\vu$ is much smaller when discretized.

It is worth mentioning that OMT in formulation \eqref{Kan1} is very close to the problems in compressed sensing. Its objective function is homogeneous degree one and the constraint is linear. It can be observed that for characterizing the OMT,
divergence operator in \eqref{Kan1} play the key roles. Later on, we extend the definition of $L_1$ OMT problem by extending these differential operators into a general meaning.

The optimization problem \eqref{Kan1} has the following dual problem:
\begin{equation}
S(\lambda^0,\lambda^1)=
\left(
\begin{array}{ll}
\underset{\phi}{\text{maximize}}& \int_{\Omega}
 \phi(\vx)(\lambda^1(\vx)-\lambda^0(\vx))\;d\vx\\
\mbox{subject to} &\|\nabla_\vx \phi(\vx)\|_{*} \le 1
\quad\text{for all } \vx\in \Omega,
\end{array}
\right)
\label{somt-dual}
\end{equation}
where the optimization and $\phi:\Omega \rightarrow \mathbb{R}$ is a function.
We write $\|\cdot\|_*$ for the dual norm of $\|\cdot\|$.

It is well-known that,
strong duality holds between
\eqref{Kan1} and \eqref{somt-dual}
in the sense that the minimized
and maximized objective values are equal
\cite{Vil03}.
Therefore, 
we take either \eqref{Kan1} or \eqref{somt-dual} as the definition of $S$.


Rigorous definitions
of the optimization problems
\eqref{Kan1} or \eqref{somt-dual}
are somewhat technical.
We skip this discussion,
as scalar optimal mass transport is standard.
In Section~\ref{s-duality-proof},
we rigorously discuss 
the vector-OMT problems,
so any rigorous discussion of the scalar-OMT problems can be inferred
as a special case.


\subsection{Theoretical properties}
\label{ss:somt-theory}

The algorithm we present in Section~\ref{sec:algorithm} is a primal-dual algorithm and, as such, finds solutions to both
Problems~\eqref{Kan1} and \eqref{somt-dual}.
This is well-defined as 
both the primal and dual problems have
solutions
\cite{Vil03}.



Write $\mathbb{R}_+$ for the set of nonnegative real numbers.
Write $\mathcal{P}(\Omega,\mathbb{R})$ for the space of nonnegative densities supported on $\Omega$ with unit mass.
We can use $S(\lambda^0,\lambda^1)$ 
as a distance measure between $\lambda^0,\lambda^1\in \mathcal{P}(\Omega,\mathbb{R})$.
The value $S:\mathcal{P}(\Omega,\mathbb{R})\times \mathcal{P}(\Omega,\mathbb{R})\rightarrow \mathbb{R}_+$
defines a metric on $\mathcal{P}(\Omega,\mathbb{R})$
\cite{Vil03}.


%
%

\section{Vector optimal mass transport}\label{sec:vectorOMT}
Next we discuss the vector-valued optimal transport, proposed recently in \cite{CheGeoTan17}.
The basic idea is to combine 
scalar optimal mass transport with network flow problems \cite{AhuMagOrl93}.

Color image processing is the simplest application to think of for this setup.
On an RGB color image, there are 3 values associated with
each pixel.
To generalize scalar-OMT to this setup, the cost for changing colors
is defined with a graph.
The edges represent the cost of changing one channel/color to another.

\subsection{Gradient and divergence on graphs}\label{sec:vectorgrad}
Consider a connected, positively weighted, undirected graph $\cG $ with $k$ nodes and $\ell$ edges. 
To define an incidence matrix for $\cG$, we say an edge $\{i,j\}$ points from $i$ to $j$,
i.e., $i\rightarrow j$, if $i<j$.
This choice is arbitrary and does not affect the final result.
With this edge orientation, the incidence matrix $D\in \mathbb{R}^{k\times \ell}$ is
\[
D_{ie}=\left\{ \begin{array}{ll}
                    +1 & \text{if edge $e=\{i,j\}$ for some node $j>i$}\\
                    -1 &  \text{if edge $e=\{j,i\}$ for some node $j<i$}\\
                    0 & \text{otherwise}.
                  \end{array}\right.
\]
For example, the incidence matrix of the graph of Figure~\ref{fig:graph-example} is
\[
D=
\begin{bmatrix}
1&1&0&1&0\\
-1&0&1&0&0\\
0&-1&-1&0&1\\
0&0&0&-1&-1
\end{bmatrix}.
\]
Write 
$\Delta_\cG =-D\diag \{1/c_1^2, \cdots, 1/c_\ell^2\} D^T$
for the (negative) graph Laplacian, where $1/c_1^2, \cdots, 1/c_\ell^2$ are the edge weights.
The edge weights are defined so that 
and $c_j$ represents the cost of traversing edge $j$ for $j=1,\dots,\ell$.

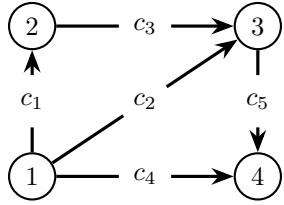
\begin{figure}
\begin{center}
\begin{tikzpicture}
\begin{scope}[every node/.style={circle,thick,draw}]
    \node (A) at (0,0) {1};
    \node (B) at (0,2) {2};
    \node (C) at (3,2) {3};
    \node (D) at (3,0) {4};
\end{scope}

\begin{scope}[>={Stealth[black]},
              every node/.style={fill=white,circle},
              every edge/.style={draw=black,very thick}]
    \path [->] (A) edge node {$c_1$} (B);
    \path [->] (A) edge node {$c_2$} (C);
    \path [->] (B) edge node {$c_3$} (C);
    \path [->] (A) edge node {$c_4$} (D);
    \path [->] (C) edge node {$c_5$} (D);
\end{scope}
\end{tikzpicture}
\end{center}
\caption{Example graph with $k=4$ nodes and $\ell=5$ edges.
To make $c_j$ the cost of traversing edge $j$,
the edge weight is defined to be $1/c_j^2$ for $j=1,\dots,\ell$.
}
\label{fig:graph-example}
\end{figure}

We define the gradient operator on $\cG$ as
$ \nabla_\cG x=  \diag \{1/c_1, \cdots,1/c_\ell\}D^T x$
and the divergence operator as
$\divg_\cG y = -D \diag \{1/c_1, \cdots,1/c_\ell\}y$.
So  the Laplacian can be rewritten as
$\Delta_\cG =\divg_\cG\nabla_\cG$.
Note that $\divg_\cG=-\nabla_\cG^*$, where $\nabla_\cG^*$ is the adjoint of $\nabla_\cG$.
This is in analogy with the usual spatial gradient and divergence operators
\cite{Maa11,ChoHuaLiZho12,ChoLiZho17,CheGeoTan17}.

The edge weights of a graph $\cG$ should be considered modeling parameters. 
In color image processing, for example, the edge weights of $\cG$ represent the cost of changing one color to another,
and they should be tuned to make results visually look best.

\subsection{Vector optimal mass transport}
We say $\vect{\lambda}:\Omega\rightarrow \mR^k_+$
is a nonnegative vector-valued density with unit mass if 
\[
\vect{\lambda}
(\vx)=
\begin{bmatrix}
\lambda_1(\vx)\\\vdots\\
\lambda_k(\vx)
\end{bmatrix},
\qquad
	\int_\Omega\sum_{i=1}^k \lambda_i(\vx)\;d\vx=1.
\]
Assume $\vect{\lambda}^0$ and $\vect{\lambda}^1$ are nonnegative vector-valued densities supported on $\Omega$ with unit mass.

We define the optimal mass transport between vector-valued densities as
    \begin{equation}\label{vomt}
V(\vect{\lambda}^0,\vect{\lambda}^1)=
\left(
\begin{array}{ll}
\underset{\vect{\vu},\vect{w}}{\text{minimize}}&
\int_{\Omega}\|\vect{\vu}(\vx)\|_u +\alpha\|\vect{w}(\vx)\|_w\;d\vx\\
\mbox{subject to} &
\divg_\vx (\vect{\vu})(\vx)
+
\divg_\mathcal{G}(\vect{w}(\vx))=\vect{\lambda}^0(\vx)-\vect{\lambda}^1(\vx)\\
& \vect{\vu} \text{ satisfies zero-flux b.c}.
\end{array}
\right)
\end{equation}
where 
 $\vect{\vu}:\Omega\rightarrow \mR^{k\times d}$
 and $\vect{w}:\Omega\rightarrow \mR^{\ell}$
 are the optimization variables, $\alpha>0$ is a parameter, and
 $\|\cdot\|_u$ is a norm on $\mathbb{R}^{k\times d}$ and $\|\cdot\|_w$ is a norm on $\mathbb{R}^{\ell}$.
 The parameter $\alpha$ represents the relative importance of the two flux terms $\vect{\vu}$ and $\vect{w}$.
 We write
 \[
\vect{\vu}=
		\begin{bmatrix}
		\vu_1^T
		\\ \vdots
		\\ \vu_k^T
		\end{bmatrix}
\qquad
\vect{w}=\begin{bmatrix}
w_1\\\vdots\\w_\ell\end{bmatrix}
\qquad
		\divg_\vx (\vect{\vu})=
		\begin{bmatrix}
		\divg_\vx (\vu_1)
		\\ \vdots
		\\ \divg_\vx(\vu_k)
		\end{bmatrix}.
 \]
 We call $\divg_\vx$ the 
  spatial divergence operator.
 The zero-flux boundary condition is 
 \[
 \vu_i(\vx)^T \mathbf{n}(\vx)=0,\,\,
\mbox{for all }\begin{cases}
\vx\in \partial \Omega,\\
\text{$ \mathbf{n}(\vx)$ normal to $\partial\Omega$}.
\end{cases}
\]
 for $i=1,\dots,k$.  
 Note that $\vect{w}$ has no boundary conditions.

The optimization problem \eqref{vomt} has the following dual problem:
\begin{equation}
V(\vect{\lambda}^0,\vect{\lambda}^1)=
\left(
\begin{array}{ll}
\underset{\vect{\phi}}{\text{maximize}}& \int_{\Omega}
\langle \vect{\phi}(\vx),\vect{\lambda}^1(\vx)
 		-\vect{\lambda}^0(\vx)\rangle \;d\vx\\
\mbox{subject to} &\|\nabla_\vx \vect{\phi}(\vx)\|_{u*} \le 1\\
&\|\nabla_\cG \vect{\phi}(\vx)\|_{w*} \le \alpha
\quad\text{for all } \vx\in \Omega,
\end{array}
\right)
\label{vomt-dual}
\end{equation}
where the optimization variable 
 $\vect{\phi}:\Omega \rightarrow \mathbb{R}^k$ is a function.
We write $\|\cdot\|_{u*}$ and $\|\cdot\|_{w*}$
for the dual norms of $\|\cdot\|_{u}$ and $\|\cdot\|_{w}$, respectively.
%

As stated in Theorem~\ref{thm:vomt-strong-duality},
strong duality holds between
\eqref{vomt} and \eqref{vomt-dual}
in the sense that the minimized
and maximized objective values are equal.
Therefore, 
we take either \eqref{vomt} or \eqref{vomt-dual} as the definition of $V$.
In Section~\ref{s-duality-proof},
we rigorously define the primal and dual problems and prove Theorem~\ref{thm:vomt-strong-duality}.

%



%
%
%
%
\subsection{Theoretical properties}
 The algorithm we present in Section~\ref{sec:algorithm} is a primal-dual algorithm and, as such, finds solutions to both
 Problems~\eqref{vomt} and \eqref{vomt-dual}.
 This is well-defined as both the primal
 and dual problems have solutions.
\begin{theorem} 
\label{thm:vomt-strong-duality}
The (infinite dimensional) primal
and dual optimization problems
\eqref{vomt} and \eqref{vomt-dual}
have solutions,
and their optimal values are the same,
i.e., strong duality holds.
\end{theorem}

Write $\mathcal{P}(\Omega,\mathbb{R}^k)$ for the space of nonnegative vector-valued densities supported on $\Omega$ with unit mass.
We can use $V(\vect{\lambda}^0,\vect{\lambda}^1)$ 
as a distance measure between $\vect{\lambda}^0,\vect{\lambda}^1\in \mathcal{P}(\Omega,\mathbb{R}^k)$.
\begin{theorem}
\label{thm:vomt-duality}
$V:\mathcal{P}(\Omega,\mathbb{R}^k)\times\mathcal{P}(\Omega,\mathbb{R}^k)\rightarrow \mathbb{R}_+$
defines a metric on $\mathcal{P}(\Omega,\mathbb{R}^k)$.
\end{theorem}

%
%
%

\section{Quantum gradient operator and matrix optimal mass transport}\label{sec:matrixOMT}
We closely follow the treatment in \cite{CheGeoTan16}. In particular, we define a notion of gradient on the space of Hermitian matrices and its dual, i.e., the (negative) divergence. 

Some applications of matrix-OMT,
such as diffusion tensor imaging,
have real-valued data
while some applications,
such as multivariate spectral analysis,
have complex-valued data \cite{stoica2005}.
To accommodate the wide range of applications,
we develop the matrix-OMT with complex-valued matrices.

Write $\cC$, $\cH$, and $\cS$ for the set of  $k\times k$ 
complex, Hermitian, and skew-Hermitian matrices respectively.
We write
$\cH_+$ for the set of $k\times k$ positive semidefinite Hermitian matrices, i.e.,
$M\in \cH_+$ if
$v^*Mv\ge 0$ for all $v\in \mathbb{C}^k$.
Write $\trace$ for the trace, i.e. for any $M\in \cH$, we have
$\trace(M)=\sum^k_{i=1}M_{ii}$.

Write $\cC^N$ for the block-column concatenation of $N$ matrices in $\cC$,
i.e., $\mathbf{Z}\in \cC^N$ if
$
\mathbf{Z}=
[Z_1^*\cdots Z_N^*]^*
$
and $Z_1,\dots,Z_N\in \cC$.
Define $\cH^N$ and $\cS^N$ likewise.
For $X,Y\in \cC$, we use the 
Hilbert-Schmidt inner product
\[
\langle X,Y\rangle=\textrm{Re}\trace(XY^*)=
\sum^k_{i=1}
\sum^k_{j=1}(\textrm{Re}X_{ij}\textrm{Re}Y_{ij}+\textrm{Im}X_{ij}\textrm{Im}Y_{ij}).
\]
(This is the standard inner product when 
we view $\cC$ as the
real vector space $\mathbb{R}^{2n^2}$.)
For $X\in \cC$, we use the norm
$\|X\|_2=(\langle X,X\rangle)^{1/2}$.
For $\mathbf{X}, \mathbf{Y}\in \cC^N$, we use the inner product
	$\langle \mathbf{X},\mathbf{Y}\rangle=\sum_{s=1}^N \langle X_s,Y_s\rangle$.
	

\subsection{Quantum gradient and divergence operators}
We define the gradient operator, given
a $\mathbf{L}=[L_1,\cdots,L_\ell]^*\in\cH^\ell$, as
	\begin{equation*}
		\nabla_\mathbf{L}: \cH \rightarrow {\cS}^\ell, ~~X \mapsto
		\left[ \begin{array}{c}
		L_1 X-X L_1\\
		\vdots \\
		L_\ell X-X L_\ell
		\end{array}\right].
	\end{equation*}
Define the divergence operator as
	\begin{equation*}
		\divg_\mathbf{L}: {\cS}^\ell \rightarrow \cH,~~Z=
		\left[ \begin{array}{c}
		Z_1\\
		\vdots \\
		Z_\ell
		\end{array}\right]
		\mapsto
		\sum_{s=1}^\ell -L_s Z_s+Z_s L_s.
	\end{equation*}
Note that $\divg_\mathbf{L}=-\nabla_\mathbf{L}^*$, where
$\nabla_\mathbf{L}^*$ is the adjoint of $\nabla_\mathbf{L}$.
This is in analogy with the usual spatial gradient and divergence operators.
Write $\Delta_\mathbf{L}=\divg_\mathbf{L}\nabla_\mathbf{L}$.
A standing assumption throughout, is that the null space of $\nabla_\mathbf{L}$, denoted by ${\rm ker}(\nabla_\mathbf{L})$, contains only scalar multiples of the identity matrix $I$.



The choice of $\mathbf{L}$ affects $\nabla_\mathbf{L}$ and, in turn, affects the induced metric.
This definition of gradient operator is inspired by the Lindblad equation in Quantum Mechanics \cite{CheGeoTan16}. There has been some work on similar matrix optimal transport theories which focus on the applications in physics \cite{CARLEN20171810,Mittnenzweig2017}.
However, how to appropriately choose $\mathbf{L}$ to fit a specific engineering application
is still a topic of open research, and there is currently no established
standard choice. 


\subsection{Matrix optimal mass transport}
We say $\Lambda:\Omega\rightarrow \cH_+$
is a nonnegative matrix-valued density with unit mass if
\[
	\int_\Omega \trace(\Lambda(\vx))\;d\vx=1.
\]
Assume $\Lambda^0$ and $\Lambda^1$ are nonnegative matrix-valued densities supported on $\Omega$ with unit mass.

We define the optimal mass transport between matrix-valued densities as
    \begin{equation}\label{momt}
M(\Lambda^0,\Lambda^1)=
\left(
\begin{array}{ll}
\underset{\vU,\vW}{\text{minimize}}&
\int_{\Omega} \|\vU(\vx)\|_u +\alpha\|\vW(\vx)\|_w\;d\vx\\
\mbox{subject to} &
\divg_\vx (\vU)(\vx)
+
\divg_\vL(\vW(\vx))=\Lambda^0(\vx)-\Lambda^1(\vx)\\
&\vU \text{ satisfies zero-flux b.c.}
\end{array}
\right)
\end{equation}
where 
 $\vU:\Omega\rightarrow \cH^{d}$
 and $\vW:\Omega\rightarrow \cS^{\ell}$
 are the optimization variables,
$\alpha>0$ is a parameter,
and $\|\cdot\|_u$ is a norm on $\cH^{d}$ and $\|\cdot\|_w$ is a norm on $\cS^{\ell}$.
The parameter $\alpha$ represents the relative importance of the two flux terms $\vU$ and $\vW$.
 We write
 \[
 \vU=
 \begin{bmatrix}
 U_1\\\vdots\\U_d
 \end{bmatrix}
 \qquad
 \vW=
\begin{bmatrix}
 W_1\\\vdots\\W_\ell
 \end{bmatrix}
 \quad
 \vu_{ij}=
 \begin{bmatrix}
( U_1)_{ij}\\
\vdots\\
( U_d)_{ij}
 \end{bmatrix}.
 \]
We define the spatial divergence as
\[
\divg_\vx(\vU)=
 \begin{bmatrix}
 \divg_\vx(\vu_{11})& \divg_\vx(\vu_{12})&\cdots&\divg_\vx(\vu_{1k})\\
\divg_\vx(\overline{ \vu}_{12})&\ddots&&\vdots\\
\vdots &&\ddots&\vdots\\
\divg_\vx(\overline{\vu}_{1k})&\divg_\vx(\overline{\vu}_{2k})&\cdots &\divg_\vx(\vu_{k,k})
 \end{bmatrix}.
 \]
 The zero-flux boundary condition is
 \[
\vu_{ij}(\vx)
^T\mathbf{n}(\vx)=0,\,\,
\mbox{for all }\begin{cases}
\vx\in \partial \Omega,\\
\text{$ \mathbf{n}(\vx)$ normal to $\partial\Omega$}.
\end{cases}
\]
 for $i,j=1,\dots,k$.
Note that $\vW$ has no boundary conditions
\cite{CheGeoNinTan17}.

The optimization problem \eqref{momt} has the following dual problem:
\begin{equation}
M(\Lambda^0,\Lambda^1)=
\left(
\begin{array}{ll}
\underset{\Phi}{\text{maximize}}& \int_{\Omega}
\langle \Phi(\vx),\Lambda^1(\vx)
 		-\Lambda^0(\vx)\rangle \;d\vx\\
\mbox{subject to} &\|\nabla_\vx \Phi(\vx)\|_{u*} \le 1\\
&\|\nabla_\mathbf{L} \Phi(\vx)\|_{w*} \le \alpha
\quad\text{for all } \vx\in \Omega,
\end{array}
\right)
\label{momt-dual}
\end{equation}
where the optimization variable
 $\Phi:\Omega \rightarrow \cH$ is a function.
We write $\|\cdot\|_{u*}$ and 
$\|\cdot\|_{w*}$ for the dual norms of 
$\|\cdot\|_{u}$ and 
$\|\cdot\|_{w}$, respectively.


As stated in Theorem~\ref{thm:momt-strong-duality}, strong duality holds between
\eqref{momt} and \eqref{momt-dual} 
in the sense that the minimized
and maximized objective values are equal.
Therefore we take 
either
\eqref{momt} or \eqref{momt-dual} 
as the definition of $M$.

In Section~\ref{ss:m-omt-rigor}, we rigorously define the primal and dual problems.
To prove the results on matrix-OMT
(Theorems~\ref{thm:momt-strong-duality} and \ref{thm:m-omt-duality})
one can take
the arguments of Section~\ref{s-duality-proof} that prove 
the results on vector-OMT
(Theorems~\ref{thm:vomt-strong-duality} and \ref{thm:vomt-duality})
and change only the notation.
We therefore simply point this out, instead of repeating same argument.

%

%
 
%
%
%
%

\subsection{Theoretical properties}
The algorithm we present in Section~\ref{sec:algorithm} is a primal-dual algorithm and, as such, finds 
solutions to both
Problems~\eqref{momt} and \eqref{momt-dual}.
This is well-defined as both the
primal and dual problems have solutions.
\begin{theorem}
\label{thm:momt-strong-duality}
The (infinite dimensional)
primal and dual optimization problems
\eqref{momt} and \eqref{momt-dual}
have solutions,
and their optimal values are the same, i.e., strong duality holds.
\end{theorem}

Write $\mathcal{P}(\Omega,\cH_+)$ for the space of nonnegative matrix-valued densities supported on $\Omega$ with unit mass.
We can use $M(\Lambda^0,\Lambda^1)$ 
as a distance measure between $\Lambda^0,\Lambda^1\in \mathcal{P}(\Omega,,\cH_+)$.
\begin{theorem} 
\label{thm:m-omt-duality}
$M:\mathcal{P}(\Omega,\cH_+)\times \mathcal{P}(\Omega,\cH_+)\rightarrow \mathbb{R}_+$
defines a metric on $\mathcal{P}(\Omega,\cH_+)$.
\end{theorem}
%
%
%
%
%
%

%
%
%
%
%

\section{Duality proof}
\label{s-duality-proof}
In this section, we establish the theoretical results.
For notational simplicity, we only prove the results for
the vector-OMT primal and dual problems \eqref{vomt} and \eqref{vomt-dual}.
Analogous results for the matrix-OMT primal and dual problems \eqref{momt} and \eqref{momt-dual}
follow from the same logic.

Although the classical scalar-OMT literature is very rich, 
standard techniques for proving scalar-OMT duality do not simply apply to our setup.
For example, Villani's proof of strong duality, presented as Theorem~1.3 of \cite{Vil03},
relies on and works with the linear optimization formulation \eqref{Monge}.
However, our vector and matrix-OMT formulations
directly generalize the flux formulation \eqref{Kan1}
and do not have formulations analogous to \eqref{Monge}.
We need a direct approach to analyze duality between the flux formulation
and the dual (with one function variable), and we provide this in this section.

We further assume
$\Omega\subset \rd$ 
has a piecewise smooth boundary. Write $\Omega^\circ$ and $\partial \Omega$ for the interior and boundary of $\Omega$.
For simplicity, assume $\Omega$ as full affine dimensions, i.e., $\overline{\Omega^\circ}=\Omega$.

The rigorous form of the dual problem \eqref{vomt-dual} is
\begin{equation}\label{vomt-dual2}
\begin{array}{ll}
\underset{\vect{\phi}\in W^{1,\infty}(\Omega,\rk)}{\text{maximize}}& \int_{\Omega}
\langle \vect{\phi}(\vx),\vect{\lambda}^1(\vx)
 		-\vect{\lambda}^0(\vx)\rangle \;d\vx\\
\mbox{subject to} &\esssup_{\vx\in \Omega}\|\nabla_\vx \vect{\phi}(\vx)\|_{u*} \le 1\\
&\sup_{\vx\in \Omega}\|\nabla_\cG \vect{\phi}(\vx)\|_{w*} \le \alpha,
\end{array}
\end{equation}
where $W^{1,\infty}(\Omega,\rk)$ is the standard Sobolev space of
functions from $\Omega$ to $\rk$ with bounded weak gradients.
That \eqref{vomt-dual2} has a solution directly follows from the Arzel\`a-Ascoli Theorem.

To rigorously define the primal problem \eqref{vomt} requires more definitions,
and we do so later as \eqref{vomt3}.


\subsection{Fenchel-Rockafellar duality}
Let $L:X\rightarrow Y$ be a continuous linear map between locally convex topological vector spaces $X$ and $Y$,
and let
$f:X\rightarrow\mathbb{R}\cup\{\infty\}$ and $g:Y\rightarrow\mathbb{R}\cup\{\infty\}$ be lower-semicontinuous convex functions. 
Write
\begin{align*}
d^\star=\sup_{x\in X}\{-f(x)-g(Lx)\}\qquad
p^\star=\inf_{y^*\in Y^*}\{f^*(L^*y^*)+g^*(-y^*)\},
\end{align*}
where
\begin{align*}
f^*(x^*)=\sup _{x\in X}\{\langle x^*,x\rangle-f(x)\}\qquad
g^*(y^*)=\sup _{y\in Y}\{\langle y^*,y\rangle-g(y)\}.
\end{align*}

In this framework of Fenchel-Rockafellar duality, 
$d^\star\le p^\star$, i.e., weak duality, holds unconditionally, and this is not difficult to prove.
On the other hand, $d^\star= p^\star$, i.e., strong duality, requires additional assumptions and is more difficult to prove.
The following theorem does this with a condition we can use.
\begin{theorem}\label{thm:duality}[Theorem~17 and 18 \cite{rockafellar1974}]
If there is an $x\in X$ such that $f(x)<\infty$ and $g$ is bounded above in a neighborhood of $Lx$,
then $p^\star=d^\star$.
Furthermore, if $p^\star=d^\star<\infty$, 
the infimum of $\inf_{y^*\in Y^*}\{f^*(L^*y^*)+g^*(-y^*)\}$ is attained.
\end{theorem}

\subsection{Spaces}
Throughout this section, $\|\cdot\|_1$, $\|\cdot\|_2$, and $\|\cdot\|_3$ denote unspecified finite dimensional norms.
As all finite dimensional norms are equivalent, we do not bother to precisely specify which norms they are.


Define
\[
C(\Omega,\rk)=
\left\{\vect{\phi}:\Omega\rightarrow \rk\,\,\Big|\,\,
\vect{\phi}\text{ is continuous},\,
\|\vect{\phi}\|_{\infty}=\max_{\vx\in \Omega}\|\vect{\phi}(\vx)\|_1<\infty
\right\}.
\]
Then $C(\Omega,\rk)$ is a Banach space 
equipped with the norm $\|\cdot \|_{\infty}$.
We define $C(\Omega,\rkd)$ likewise.
If $\vect{\phi}\in C(\Omega,\rk)$ is continuously differentiable, 
$\nabla_\vx \vect{\phi}$ is defined on $\Omega^\circ$.
We say $\nabla_\vx\vect{\phi}$ has a continuous extension to $\Omega$,
if there is a $g\in C(\Omega,\rkd)$ such that
$g|_{\Omega^\circ}=\nabla_\vx \vect{\phi}$.
Define
\begin{align*}
C^1(\Omega,\rk)=
\Big\{\vect{\phi}:\rd\rightarrow \rk\,\,\big|\,\,&
\vect{\phi}\text{ is continuously differentiable on }\Omega^\circ,\\
&\nabla_\vx \vect{\phi}\text{ has a continuous extension to }\Omega,\\
&\|\vect{\phi}\|_{\infty,\infty}=\max_{\vx\in \Omega}\|\vect{\phi}(\vx)\|_2
+\sup_{\vx\in \Omega}\|\nabla_\vx\vect{\phi}(\vx)\|_3<\infty
\Big\}.
\end{align*}
Then $C^1(\Omega,\rk)$ is a Banach space 
equipped with the norm $\|\cdot\|_{\infty,\infty}$.

Write $\mathcal{M}(\Omega,\rk)$ for the space of $\rk$-valued signed finite Borel measures on $\Omega$,
and define $\mathcal{M}(\Omega,\rkd)$ likewise.
Write $(C(\Omega,\rkd))^*$, $(C^1(\Omega,\rk))^*$ for the topological dual of 
$C(\Omega,\rkd)$, $C^1(\Omega,\rk)$, respectively.
The standard Riesz-Markov theorem tells us that
$(C(\Omega,\rkd))^*=\mathcal{M}(\Omega,\rkd)$.

Fully characterizing $(C^1(\Omega,\rk))^*$ is hard, but we do not need to do so.
Instead, we only use the following simple fact.
Any $g\in \mathcal{M}(\Omega,\rk)$ defines the bounded linear map
$\vect{\phi}\mapsto \int_\Omega \langle\vect{\phi}(\vx),g(d\vx)\rangle$
for any $\vect{\phi}\in C^1(\Omega,\rk)$. In other words, 
$\mathcal{M}(\Omega,\rkd)\subset(C^1(\Omega,\rk))^*$ with the appropriate identification.

\subsection{Operators}
We redefine 
$\nabla_\vx:C^1(\Omega,\rk)\rightarrow C(\Omega,\rkd)$
so that $\nabla_\vx\vect{\phi}$ is the continuous extension of the usual $\nabla_\vx\vect{\phi}$ to all of $\Omega$.
This makes $\nabla_\vx$ a bounded linear operator.
Define the dual (adjoint) operator
$\nabla_\vx^*:\mathcal{M}(\Omega,\rkd)\rightarrow (C^1(\Omega,\rk))^*$ by
\[
\int_\Omega \langle\vect{\phi}(\vx),(\nabla_\vx^*\vect{\vu})(d\vx)\rangle
=
\int_\Omega \langle(\nabla_\vx \vect{\phi})(\vx),\vect{\vu}(d\vx)\rangle
\]
for any $\vect{\phi}\in C^1(\Omega,\rk)$ and $\vect{\vu}\in \mathcal{M}(\Omega,\rkd)$.

Define the $\nabla_\cG$ (which is simply a multiplication by a $\mathbb{R}^{k\times\ell}$ matrix) as
$\nabla_\cG:C^1(\Omega,\rk)\rightarrow C(\Omega,\mathbb{R}^\ell)$.
Since $C^1(\Omega,\mathbb{R}^\ell)\subset C(\Omega,\mathbb{R}^\ell)$, there is nothing wrong with defining the range of $\nabla_\cG$
to be $C(\Omega,\mathbb{R}^\ell)$,
and this still makes $\nabla_\cG$ a bounded linear operator.
Define the dual (adjoint) operator
$\nabla_\cG^*:\mathcal{M}(\Omega,\mathbb{R}^\ell)\rightarrow (C^1(\Omega,\rk))^*$ 
by identifying $\nabla_\cG^*$ with the transpose of the matrix that defines $\nabla_\cG$.
Since $\nabla_\cG^*$ is simply multiplication by a matrix,
we can further say
\[
\nabla_\cG^*:\mathcal{M}(\Omega,\mathbb{R}^\ell)\rightarrow \mathcal{M}(\Omega,\rk)\subset(C^1(\Omega,\rk))^*.
\]
We write $\divg_\cG=-\nabla_\cG^*$.

\subsection{Zero-flux boundary condition}
Let $\vect{\mathbf{m}}:\Omega\rightarrow \rkd$ a smooth function.
Then integration by parts tells us that 
\[
\int_\Omega\langle\nabla_\vx\vect{\varphi}(\vx),\vect{\mathbf{m}}(\vx)\rangle\;d\vx=
-\int_\Omega\langle \vect{\varphi}(\vx),\divg_\vx\vect{\mathbf{m}}(\vx)\rangle\;d\vx
\]
holds for all smooth $\vect{\varphi}:\Omega\rightarrow \rk$ if and only if
$\vect{\mathbf{m}}(\vx)$ satisfies the zero-flux boundary condition, i.e.,
$\vect{\mathbf{m}}(\vx)\mathbf{n}(\vx)=0$ for all $\vx\in \partial \Omega$,
where $\mathbf{n}(\vx)$ denotes the normal
vector at $\vx$.
Here $\divg_\vx$ denotes the usual (spatial) divergence.
To put it differently, 
$\nabla_\vx^*=-\divg_\vx$ holds when the zero-flux boundary condition holds.


We generalize this notion to measures. We say 
$\vect{\vu}\in \mathcal{M}(\Omega,\rkd)$
satisfies the zero-flux boundary condition in the weak sense if
there is a $\vect{g}\in \mathcal{M}(\Omega,\rk)\subset(C^1(\Omega,\rk))^* $ such that
\[
\int_\Omega\langle\nabla_\vx\vect{\phi} (\vx),\vect{\mathbf{u}}(d\vx)\rangle=
-\int_\Omega\langle \vect{\phi}(\vx),\vect{g}(d\vx)\rangle
\]
holds for all $\vect{\phi}\in C^1(\Omega,\rk)$.
In other words, $\vect{\vu}$ satisfies the zero-flux boundary condition 
if $\nabla_\vx^*\vect{\vu}\in\mathcal{M}(\Omega,\rk)\subset(C^1(\Omega,\rk))^*$.
In this case, we write $\divg_\vx(\vect{\vu})=\vect{g}$
and $\divg_\vx(\vect{\vu})=-\nabla_\vx^*(\vect{\vu})$.
This definition is often used in elasticity theory.

\subsection{Duality}
To establish duality, we view the dual problem
\eqref{vomt-dual}
as the primal problem and obtain the primal problem
\eqref{vomt} as the dual of the dual.
We do this because the dual of $C(\Omega,\rkd)$ is known, while the 
dual of $\mathcal{M}(\Omega,\rkd)$ is difficult to characterize.

Consider the problem
\begin{equation}
\begin{array}{ll}
\underset{\vect{\phi}\in C^1(\Omega,\rk)}{\text{maximize}}& \int_{\Omega}
\langle \vect{\phi}(\vx),\vect{\lambda}^1(\vx)
 		-\vect{\lambda}^0(\vx)\rangle \;d\vx\\
\mbox{subject to} &\|\nabla_\vx \vect{\phi}(\vx)\|_{u*} \le 1\\
&\|\nabla_\cG \vect{\phi}(\vx)\|_{w*} \le \alpha
\quad\text{for all } \vx\in \Omega,
\end{array}
\label{vomt-dual3}
\end{equation}
which is equivalent to \eqref{vomt-dual2}.
Define 
\begin{align*}
L:C^1(\Omega,\rk)&\rightarrow C(\Omega,\rkd)\times C(\Omega,\mathbb{R}^\ell)\\
\vect{\phi}&\mapsto (\nabla_\vx\vect{\phi},\nabla_\cG\vect{\phi})
\end{align*}
and
\[
g(a,b)=\left\{
\begin{array}{ll}
0&\text{if } \|a(\vx)\|_{u*}\le1,\,\|b(\vx)\|_{w*}\le\alpha\text{ for all }\vx\in \Omega\\
\infty&\text{otherwise}.
\end{array}
\right.
\]
Rewrite \eqref{vomt-dual3} as
\[
\begin{array}{ll}
\underset{\vect{\phi}\in C^1(\Omega,\rk)}{\text{maximize}}& \int_{\Omega}
\langle \vect{\phi}(\vx),\vect{\lambda}^1(\vx)
 		-\vect{\lambda}^0(\vx)\rangle \;d\vx
 		-g(L\vect{\phi}),
\end{array}
\]
and consider its Fenchel-Rockafellar dual
\begin{equation*}
\begin{array}{ll}
\underset{
\substack{
\vect{\vu}\in \mathcal{M}(\Omega,\rkd)\\
\vect{w}\in \mathcal{M}(\Omega,\rk)
}
}{\text{minimize}}&
\int_{\Omega}\|\vect{\vu}(\vx)\|_u +\alpha\|\vect{w}(\vx)\|_w\;d\vx\\
\mbox{subject to} &
-\nabla_\vx^*\vect{\vu}-\nabla_\cG^*\vect{w}=\vect{\lambda}^0-\vect{\lambda}^1
\text{ as members of }(C^1(\Omega,\rk))^*.
\end{array}
\end{equation*}

The constraint
\[
-\nabla_\vx^*\vect{\vu}=\vect{\lambda}^0-\vect{\lambda}^1-\divg_\cG\vect{w}\in\mathcal{M}(\Omega,\rk)\subset (C^1(\Omega,\rk))^*
\]
implies 
\[
-\nabla_\vx^*\vect{\vu}\in \mathcal{M}(\Omega,\rk),
\]
i.e., $\vect{\vu}$ satisfies the zero-flux boundary condition.

We now state the rigorous form of the primal problem \eqref{vomt}
\begin{equation}\label{vomt3}
\begin{array}{ll}
\underset{
\substack{
\vect{\vu}\in \mathcal{M}(\Omega,\rkd)\\
\vect{w}\in \mathcal{M}(\Omega,\rk)
}
}{\text{minimize}}&\int_{\Omega}\|\vect{\vu}(\vx)\|_u +\alpha\|\vect{w}(\vx)\|_w\;d\vx\\
\mbox{subject to} &
\divg_\vx\vect{\vu}+\divg_\cG\vect{w}=\vect{\lambda}^0-\vect{\lambda}^1
\text{ as members of }\mathcal{M}(\Omega,\rk)\\
& \vect{\vu} \text{ satisfies zero-flux b.c in the weak sense}.
\end{array}
\end{equation}

The point $\vect{\phi}=0$ satisfies the assumption of Theorem~\ref{thm:duality}.
Furthermore, it is easy to verify that the optimal value of the dual problem \eqref{vomt-dual} is bounded.
This implies strong duality, \eqref{vomt3} is feasible, and \eqref{vomt3} has a solution.

For $V(\vect{\lambda}^0,\vect{\lambda}^1)$ to be a metric, the 4 metric axioms must be satisfied.
These are relatively straightforward to prove, and interested readers can find the argument in \cite{CheGeoNinTan17}.
However, one more implicit axiom (not shown in previous works) must be shown:
 $V(\vect{\lambda}^0,\vect{\lambda}^1)<\infty$ for all $\vect{\lambda}^0,\vect{\lambda}^1\in \mathcal{P}(\Omega,\mathbb{R}^k)$.
 This follows from the fact that \eqref{vomt3} has a solution.
 This completes the proof that
$V:\mathcal{P}(\Omega,\mathbb{R}^k)\times \mathcal{P}(\Omega,\mathbb{R}^k)\rightarrow\mathbb{R}_+$
defines a metric.

\subsection{Formal setup of Matrix-OMT}
\label{ss:m-omt-rigor}
The rigorous form of the primal problem \eqref{momt} is 
    \begin{equation*}
\begin{array}{ll}
\underset{
\substack{\vU\in \mathcal{M}(\Omega,\cH^d)\\\vW\in \mathcal{M}(\Omega,\cS^\ell)}}{\text{minimize}}&
\int_{\Omega} \|\vU(\vx)\|_u +\alpha\|\vW(\vx)\|_w\;d\vx\\
\mbox{subject to} &
\divg_\vx (\vU)
+
\divg_\vL\vW=\Lambda^0-\Lambda^1
\text{ as members of }\mathcal{M}(\Omega,\cH)\\
&\vU \text{ satisfies zero-flux b.c in the weak sense,}
\end{array}
\end{equation*}
and the rigorous form of the dual problem \eqref{momt-dual} is
\begin{equation*}
\begin{array}{ll}
\underset{\Phi\in W^{1,\infty}(\Omega,\cH)}{\text{maximize}}& \int_{\Omega}
\langle \Phi(\vx),\Lambda^1(\vx)
 		-\Lambda^0(\vx)\rangle \;d\vx\\
\mbox{subject to} &\|\nabla_\vx \Phi(\vx)\|_{u*} \le 1\\
&\|\nabla_\mathbf{L} \Phi(\vx)\|_{w*} \le \alpha
\quad\text{for all } \vx\in \Omega,
\end{array}
\end{equation*}
where the measure and Sobolev spaces
are defined similarly.
Theorems~\ref{thm:momt-strong-duality} and \ref{thm:m-omt-duality}
follow from the exact same reasoning as that of the vector counterparts.

\section{Algorithmic preliminaries}
\label{s:alg-prelim}
%

Consider the Lagrangian for the vector optimal transport problems \eqref{vomt} and its dual \eqref{vomt-dual}
\begin{align}
L(\vect{\vu},\vect{w},\vect{\phi})=&\int_\Omega \|\vect{\vu}(\vx)\|_u+\alpha\|\vect{w}(\vx)\|_w\;d\vx
\nonumber\\
&\qquad\qquad+
\int_\Omega 
\langle\vect{\phi}(\vx),
\divg_\vx(\vect{\vu})(\vx)
+\divg_\cG(\vect{w}(\vx))
-\vect{\lambda}^0(\vx)+\vect{\lambda}^1(\vx)\rangle\;d\vx,
\label{L-vomt}
\end{align}
which is convex with respect to $\vect{\vu}$ and $\vect{w}$ and concave with respect to $\phi$.

Finding a saddle point of \eqref{L-vomt} is equivalent to solving \eqref{vomt} and  \eqref{vomt-dual},
when the primal problem \eqref{vomt} has a solution, the dual problem \eqref{vomt-dual}  has a solution,
and the optimal values of  \eqref{vomt} and  \eqref{vomt-dual}  are equal.
See \cite[Theorem~7.1]{bauschke2012}, \cite[Theorem 2]{liu2017}, or any reference on standard
convex analysis such as  \cite{rockafellar1974} for further discussion on this point.

To solve the optimal transport problems, we discretize the continuous problems and 
apply PDHG method, which we soon describe, to solve the discretized convex-concave saddle point problem.

\subsection{PDHG method}
Consider the convex-concave saddle function
\[
L(x,y,z)=f(x)+g(y)+\langle Ax+By,z\rangle-h(z),
\]
where
$f$, $g$, and $h$ are (closed and proper) convex functions and 
$x\in \mathbb{R}^n$, $y\in \mathbb{R}^m$, $z\in \mathbb{R}^l$,
$A\in \mathbb{R}^{l\times n}$, and $B\in \mathbb{R}^{l\times m}$.
Note $L$ is convex in $x$ and $y$ and concave in $z$.
Assume $L$ has a saddle point
and step sizes $\mu,\nu,\tau>0$ satisfy
\[
1>\tau\mu \lambda_\mathrm{max}(A^TA)+\tau \nu \lambda_\mathrm{max}(B^TB).
\]
Write $\|\cdot\|_2$ for the standard Euclidean norm.
Then the method
\begin{align}
x^{k+1}&=
\argmin_{x\in \mathbb{R}^n}
\left\{
L(x,y^k,z^k) +\frac{1}{2\mu}\|x-x^k\|_2^2
\right\}\nonumber\\
y^{k+1}&=
\argmin_{y\in \mathbb{R}^m}
\left\{
L(x^k,y,z^k) +\frac{1}{2\nu}\|y-y^k\|_2^2
\right\}
\label{cp-method}\\
z^{k+1}&=
\argmax_{z\in \mathbb{R}^l}
\left\{
L(2x^{k+1}-x^k,2y^{k+1}-y^k,z) -\frac{1}{2\tau}\|z-z^k\|_2^2
\right\}\nonumber
\end{align}
converges to a saddle point.
This method is called  
the Primal-Dual Hybrid Gradient (PDHG) method
or the 
(preconditioned) Chambolle-Pock method
\cite{esser2010,ChaPoc11,PockCha11}.

PDHG can be interpreted as a proximal point method
under a certain metric \cite{he2012}.
The quantity
\begin{align*}
R^k=&
\frac{1}{\mu}\|x^{k+1}-x^k\|_2^2+\frac{1}{\mu}\|y^{k+1}-y^k\|_2^2+
\frac{1}{\tau}\|z^{k+1}-z^k\|_2^2\nonumber\\
\qquad\qquad&-2\langle z^{k+1}-z^k,A(x^{k+1}-x^k)+B(y^{k+1}-y^k)\rangle.
\end{align*}
is the fixed-point residual of the non-expansive mapping 
defined by the proximal point method.
Therefore $R^k=0$ if and only if $(x^k,y^k,z^k)$
is a saddle point of $L$,
and $R^k$ decreases monotonically to $0$, cf., review paper \cite{ryu2016}.
We can use $R^k$ as a measure of progress
and as a termination criterion.

\subsection{Shrink operators}
\label{ss:shrink}
As the subproblems of PDHG \eqref{cp-method}
are optimization problems themselves,
PDHG is most effective when these subproblems have closed-form solutions.

The problem definitions of scalar, vector, and matrix-OMT involve norms. 
For some, but not all, choices of norms, the ``shrink'' operators
\begin{equation*}
\srk(x^0;\mu)=
\argmin_{x\in \mathbb{R}^n}
\left\{
 \mu\|x\|+(1/2)\|x-x^0\|_2^2
\right\}
\end{equation*}
have closed-form solutions.
Therefore, when possible, it is useful to choose such norms for computational efficiency.
Readers familiar with the compressed sensing or proximal methods literature
may be familiar with this notion.


For the vector-OMT, we focus on norms
\[
\|\vect{\vu}\|_2^2=
\sum^d_{s=1}
\|\vu_s\|_2^2
\qquad
\|\vect{\vu}\|_{1,2}=
\sum^d_{s=1}
\|\vu_s\|_2
\qquad
\|\vect{\vu}\|_1=
\sum^k_{s=1}
\|\vu_s\|_1
\]
for $\vect{\vu}\in \mathbb{R}^{k\times d}$ and 
\[
\|\vect{w}\|_2^2=
\sum^\ell_{s=1}
(w_s)^2
\qquad
\|\vect{w}\|_1=
\sum^\ell_{s=1}
|w_s|
\]
for $\vect{w}\in \mathbb{R}^{\ell}$.
The shrink operators of these norms have closed-form solutions.

For the matrix-OMT, we focus on norms
\[
\|\vU\|_2^2=
\sum^d_{s=1}
\sum^k_{i,j=1}
|(U_s)_{ij}|^2
\,
\|\vU\|_1=
\sum^d_{s=1}
\sum^k_{i,j=1}
|(U_s)_{ij}|
\quad
\|\vU\|_\mathrm{1,nuc}=
\sum^d_{s=1}
\|U_s\|_\mathrm{nuc}
\]
for $\vU\in \cH^{d}$ and 
$\|\cdot\|_2$, $\|\cdot\|_1$,  and $\|\cdot\|_{1,\mathrm{nuc}}$
for $\vW\in \cS^\ell$, which are defined likewise.
The nuclear norm $\|\cdot\|_\mathrm{nuc}$ is the sum of the singular values.
The shrink operators of these norms have closed-form solutions.

 We provide further information and details on shrink operators in the appendix.

\section{Algorithms}\label{sec:algorithm}
We now present simple and parallelizable algorithms for the OMT problems.
These algorithms are, in particular, very well-suited for GPU computing.

In Section~\ref{sec:algorithm} and \ref{sec:examples}
we use an $n\times n$ discretization of the 2D domain $[0,1]\times[0,1]$
to obtain approximate solutions to the continuous problems.
For simplicity of notation,
we use the same symbol to denote the discretized variables and their continuous counterparts.
Whether we are referring to the continuous variable or its discretization should be clear from  context.

%
%
%

As mentioned in Section~\ref{s:alg-prelim}, these methods are the PDHG method
applied to discretizations of the continuous problems.
In the implementation, it is important to get the discretization at the boundary correct
in order to respect the zero-flux boundary conditions.
For interested readers, the details are provided in the appendix.

Instead of detailing the somewhat repetitive derivations of the algorithms in full,
we simply show the key steps and arguments for the $\vect{\vu}$ update of vector-OMT.
The other steps follow from similar logic.

When we discretize the primal and dual vector-OMT problems and
apply PDHG to the discretized Lagrangian form of  \eqref{L-vomt}, we get
\begin{align*}
\vect{\vu}^{k+1}&=
\argmin_{\vect{\vu}\in \mathbb{R}^{n\times n\times k\times d}}
\left\{
\sum_{ij}
\|\vect{\vu}_{ij}\|_u
+\langle \vect{\phi}_{ij},(\divg_\vx \vu)_{ij}\rangle
 +\frac{1}{2\mu}\|\vect{\vu}_{ij}-\vect{\vu}^k_{ij}\|_2^2
\right\}\\
&=
\argmin_{\vect{\vu}\in \mathbb{R}^{n\times n\times k\times d}}
\left\{\sum_{ij}
\mu\|\vect{\vu}_{ij}\|_u
-\mu\langle (\nabla_\vx\vect{\phi})_{ij},\vu_{ij}\rangle
 +(1/2)\|\vect{\vu}_{ij}-\vect{\vu}_{ij}^k\|_2^2
\right\}.
\end{align*}
Since the minimization splits over the $i,j$ indices, we write
\begin{align*}
\vect{\vu}^{k+1}_{ij}
&=
\argmin_{\vect{\vu}_{ij}\in \mathbb{R}^{ k\times d}}
\left\{
\mu\|\vect{\vu}_{ij}\|_u
-\mu\langle (\nabla_\vx\vect{\phi})_{ij},\vu_{ij}\rangle
 +(1/2)\|\vect{\vu}_{ij}-\vect{\vu}_{ij}^k\|_2^2
\right\}\\
&=
\argmin_{\vect{\vu}_{ij}\in \mathbb{R}^{ k\times d}}
\left\{
\mu\|\vect{\vu}_{ij}\|_u
 +(1/2)\|\vect{\vu}_{ij}-(\vect{\vu}^k_{ij}+\mu(\nabla_\vx \vect{\phi})_{ij})\|_2^2
\right\}\\
&=\srk(\vect{\vu}^k_{ij}+\mu(\nabla_\vx \vect{\phi})_{ij};\mu).
\end{align*}
At the boundary, these manipulations need special care.
When we incorporate ghost cells in our discretization,
these seemingly cavalier manipulations are also correct on the boundary.
We further explain the ghost cells and discretization 
 in the appendix.

%

\subsection{Scalar-OMT algorithm}
\label{ss:somt-algorithm}
The scalar-OMT algorithm can be viewed 
as a special case of vector-OMT or matrix-OMT algorithms.
This scalar-OMT algorithm was presented in \cite{LiRyuOsh17},
but we restate it here for completeness.

\begin{tabbing}
aaaaa\= aaa \=aaa\=aaa\=aaa\=aaa=aaa\kill  
   \rule{\linewidth}{0.8pt}\\
   \noindent{\large\bf First-order Method for S-OMT}\\
  \1 \textbf{Input}: Problem data $\lambda^0$, $\lambda^1$\\
    \3 Initial guesses $\vu^0$, $\phi^0$ and step sizes $\mu$, $\tau$\\
  \1 \textbf{Output}: Optimal $\vu^\star$ and $\phi^\star$\\
\rule{\linewidth}{0.5pt}\\
  \1 \For $k=1, 2, \cdots$ \qquad \textrm{(Iterate until convergence)}\\
  \2 $\vu^{k+1}_{ij}=\srk(\vu_{ij}^k+\mu(\nabla \Phi^k)_{ij}, \mu)$ 
  \qquad for $i,j=1,\dots,n$\\
  \2  $\phi_{ij}^{k+1}=\phi_{ij}^k+\tau (\divg_\vx (2\vu^{k+1}-\vu^k)_{ij}+\lambda^1_{ij}-\lambda^0_{ij})$ \qquad for $i,j=1,\dots,n$\\ 
  \1 \End\\
   \rule{\linewidth}{0.8pt}
\end{tabbing}

This method converges for step sizes $\mu,\tau>0$ that satisfy
\[
1>\tau\mu \lambda_\mathrm{max}(-\Delta_\vx).
\]
For the particular setup of $\Omega=[0,1]\times[0,1]$ and $\Delta x=1/(n-1)$,
the bound $\lambda_\mathrm{max}(-\Delta_\vx)\le 8/(\Delta x)^2=8(n-1)^2$ is known.
In our experiments, we use $\mu =1/(16\tau(n-1)^2)$,
a choice that ensures convergence for any $\tau>0$.
We tune $\tau$ for the fastest convergence.

\subsection{Vector-OMT algorithm}
\label{ss:vomt-algorithm}
Write $\srk_u$ and $\srk_w$ for the shrink operators with respect to $\|\cdot\|_u$ and $\|\cdot\|_w$.
The vector-OMT algorithm is as follows:

\begin{tabbing}
aaaaa\= aaa \=aaa\=aaa\=aaa\=aaa=aaa\kill  
   \rule{\linewidth}{0.8pt}\\
   \noindent{\large\bf First-order Method for V-OMT}\\
  \1 \textbf{Input}: Problem data $\cG$, $\vect{\lambda}^0$, $\vect{\lambda}^1$, $\alpha$\\
    \3 Initial guesses $\vect{\vu}^0$, $\vect{w}^0$, $\vect{\phi}^0$ and step sizes $\mu$, $\nu$, $\tau$\\
  \1 \textbf{Output}: Optimal $\vect{\vu}^\star$, $\vect{w}^\star$, and $\vect{\phi}^\star$\\
\rule{\linewidth}{0.5pt}\\
  \1 \For $k=1, 2, \cdots$ \qquad \textrm{(Iterate until convergence)}\\
  \2 $\vect{\vu}^{k+1}_{ij}=\srk_u(\vect{\vu}_{ij}^k+\mu(\nabla \vect{\phi}^k)_{ij}, \mu)$ 
  \qquad \quad for $i,j=1,\dots,n$\\
\2 $\vect{w}^{k+1}_{ij}=\srk_w(\vect{w}^k_{ij}+\nu(\nabla_\cG\vect{\phi}^k_{ij}),\alpha\nu)$   \qquad for $i,j=1,\dots,n$\\
  \2  $\vect{\phi}_{ij}^{k+1}=\vect{\phi}_{ij}^k+\tau (\divg_\vx (2\vect{\vu}^{k+1}-\vect{\vu}^k)_{ij}+
  \divg_\cG (2\vect{w}^{k+1}-\vect{w}^k)_{ij}  +\vect{\lambda}^1_{ij}-\vect{\lambda}^0_{ij})$\\
\4 \qquad \qquad \qquad\qquad\qquad\qquad\qquad\qquad\qquad\qquad\qquad  for $i,j=1,\dots,n$\\ 
  \1 \End\\
   \rule{\linewidth}{0.8pt}
\end{tabbing}

This method converges for step sizes $\mu,\nu,\tau>0$  that satisfy
\[
1>\tau\mu \lambda_\mathrm{max}(-\Delta_\vx)
+\tau \nu \lambda_\mathrm{max}(-\Delta_\cG).
\]
For the particular setup of $\Omega=[0,1]\times[0,1]$ and $\Delta x=1/(n-1)$,
the bound $\lambda_\mathrm{max}(-\Delta_\vx)\le 8(n-1)^2$ is known.
Given a graph $\cG$, we can compute 
$\lambda_\mathrm{max}(-\Delta_\cG)$ with a standard eigenvalue routine.
In our experiments, we use
$\mu =1/(32\tau(n-1)^2)$
and
$\nu =1/(4\tau\lambda_\mathrm{max}(-\Delta_\cG))$,
a choice that ensures convergence for any $\tau>0$.
We tune $\tau$ for the fastest convergence.

\subsection{Matrix-OMT algorithm}
\label{ss:momt-algorithm}
Write $\srk_u$ and $\srk_w$ for the shrink operators with respect to $\|\cdot\|_u$ and $\|\cdot\|_w$.
The matrix-OMT algorithm is as follows:

\begin{tabbing}
aaaaa\= aaa \=aaa\=aaa\=aaa\=aaa=aaa\kill  
   \rule{\linewidth}{0.8pt}\\
   \noindent{\large\bf First-order Method for M-OMT}\\
  \1 \textbf{Input}: Problem data $\mathbf{L}$, $\Lambda^0$, $\Lambda^1$, $\alpha$, \\
    \3 Initial guesses $\vU^0$, $\vW^0$, $\Phi^0$ and step sizes $\mu$, $\nu$, $\tau$\\
  \1 \textbf{Output}: Optimal $\vU^\star$, $\vW^\star$, and $\Phi^\star$\\
\rule{\linewidth}{0.5pt}\\
  \1 \For $k=1, 2, \cdots$ \qquad \textrm{(Iterate until convergence)}\\
  \2 $\vU^{k+1}_{ij}=\srk_u(\vU_{ij}^k+\mu(\nabla \Phi^k)_{ij}, \mu)$ 
  \qquad\quad \,\,for $i,j=1,\dots,n$\\
\2 $\vW^{k+1}_{ij}=\srk_w(\vW^k_{ij}+\nu(\nabla_\mathbf{L}\Phi^k_{ij}),\alpha\nu)$   \qquad for $i,j=1,\dots,n$\\
  \2  $\Phi_{ij}^{k+1}=\Phi_{ij}^k+\tau (\divg_\vx (2\vU^{k+1}-\vU^k)_{ij}+
  \divg_\mathbf{L} (2\vW^{k+1}-\vW^k)_{ij}  +\Lambda^1_{ij}-\Lambda^0_{ij})$\\
\4 \qquad \qquad \qquad\qquad\qquad\qquad\qquad\qquad\qquad\qquad\qquad for $i,j=1,\dots,n$\\ 
  \1 \End\\
   \rule{\linewidth}{0.8pt}
\end{tabbing}

This method converges for step sizes $\mu,\nu,\tau>0$  that satisfy
\[
1>\tau\mu \lambda_\mathrm{max}(-\Delta_\vx)
+\tau \nu \lambda_\mathrm{max}(-\Delta_\mathbf{L}).
\]
For the particular setup of $\Omega=[0,1]\times[0,1]$ and $\Delta x=1/(n-1)$,
the bound $\lambda_\mathrm{max}(-\Delta_\vx)\le 8(n-1)^2$ is known.
Given $\mathbf{L}$, we can compute the value of
$\lambda_\mathrm{max}(-\Delta_\mathbf{L})$
by explicitly forming a $k^2\times k^2$ matrix
representing the linear operator $-\Delta_\mathbf{L}$
and applying a standard eigenvalue routine.
In our experiments, we use
$\mu =1/(32\tau(n-1)^2)$
and 
$\nu =1/(4\tau\lambda_\mathrm{max}(-\Delta_\mathbf{L}))$,
a choice that ensures convergence for any $\tau>0$.
We tune $\tau$ for the fastest convergence.

\begin{figure}[h!]
\centering
\subfloat[$\vect{\lambda}^0$]{\includegraphics[width=0.49\textwidth]{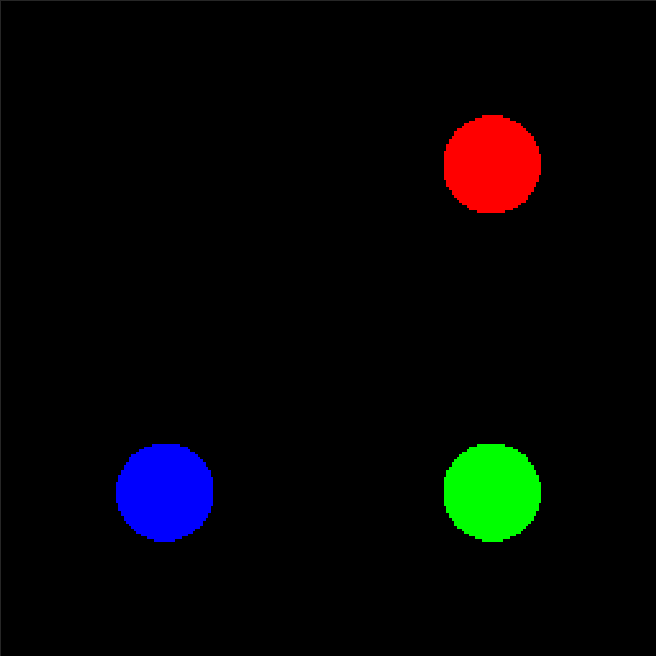}
 \label{fig:vec_marginal0}}
\subfloat[$\vect{\lambda}^1$]{\includegraphics[width=0.49\textwidth]{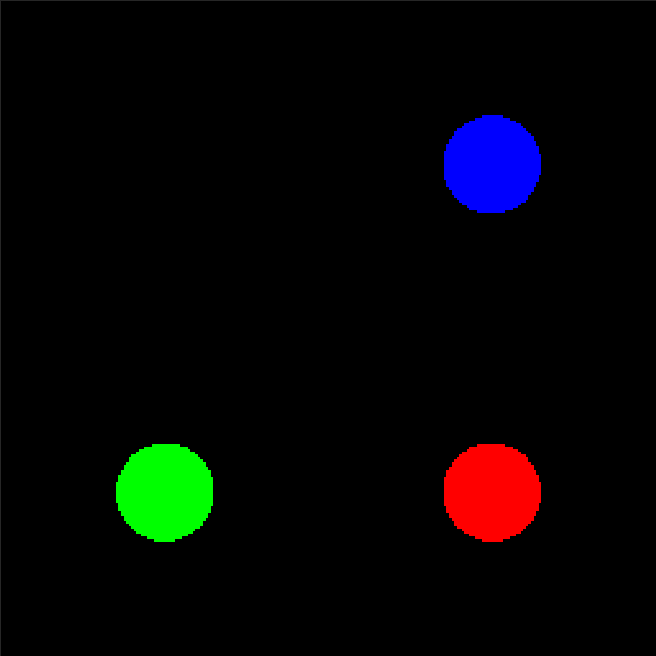}
 \label{fig:vec_marginal1}}\\
 \vspace{-0.05in}
\subfloat[Velocity field $\vu$ with $c_1=1$, $c_2=1$, $c_3=1$,
 $\alpha=1$ 
 \newline
 and  $\|\vect{\vu}\|_{1,2}, \|\vect{w}\|_{1}$. 
$V(\vect{\lambda}^0,\vect{\lambda}^1)=0.57$.
]{\includegraphics[width=0.49\textwidth]{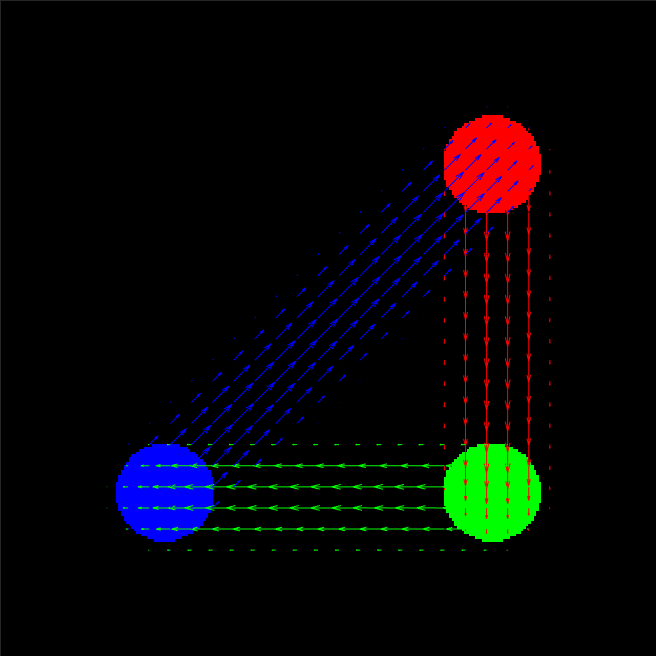}
 \label{fig:velocityuL2a}
}
\subfloat[Velocity field $\vu$ with $c_1=1$, $c_2=1$, 
 $c_3=1$, $\alpha=1$ 
 \newline and $\|\vect{\vu}\|_{1}, \|\vect{w}\|_{1}$.
$ V(\vect{\lambda}^0,\vect{\lambda}^1)=0.67$.
]{\includegraphics[width=0.49\textwidth]{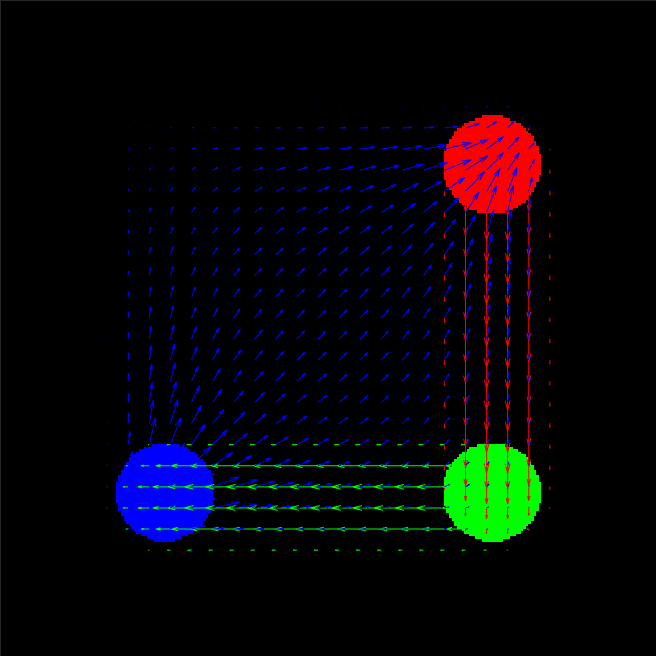}
 \label{fig:velocityuL2b}
}
\caption{Color image vector-OMT example.}
 \label{fig:velocityuL2}
\end{figure}

\subsection{Parallelization}
For the vector-OMT, 
the computation for the  $\vect{\vu}$, $\vect{w}$, and $\vect{\phi}$ updates splits over the indices $(i,j)$,
i.e., the computation splits pixel-by-pixel. Furthermore, the $\vect{\vu}$ and $\vect{w}$ updates can be done in parallel.
Parallel processors handing jobs split over
 $(i,j)$ must be synchronized before and after the $\vect{\phi}$ update.
The same is true for the scalar-OMT and matrix-OMT.

This highly parallel and regular algorithmic structure makes 
the proposed algorithms very well suited for CUDA GPUs.
We demonstrate this through our experiments.

\section{Examples}\label{sec:examples}
In this section, we provide example applications of vector and matrix-OMT
with numerical tests to demonstrate the effectiveness of our algorithms.  
As mentioned in the introduction,
potential applications of vector and matrix-OMT are broad. Here, 
we discuss two of the simplest applications.

We implemented the algorithm on C++ CUDA and ran it on a Nvidia GPU.
For convenience, we MEXed this code, i.e., the code is made available as a Matlab function.
For scientific reproducibility, we release this code.

\subsection{Color images}
Color images in RGB format is one of the more immediate examples of vector-valued densities.
At each spatial position of a 2D color image, the color is represented as a combination of the three basic colors
red (R), green (G), and blue (B). 
We allow any basic color to change to another basic color with cost
$c_1$ for R to G, $c_2$ for R to B, and $c_3$ for G to B.
So the graph $\cG$ as described in Section~\ref{sec:vectorgrad} has 3 nodes and 3 edges.

Consider the two color images on the domain $\Omega=[0,1]\times [0,1]$
with $256\times 256$ discretization shown in
Figures~\ref{fig:vec_marginal0} and \ref{fig:vec_marginal1}.
The initial and target densities $\vect{\lambda}^0$ and $\vect{\lambda}^1$
both have three disks at the same location, but with different colors.
The optimal flux depends on the choice of norms.
Figures~\ref{fig:velocityuL2a} and \ref{fig:velocityuL2b},
show fluxes optimal with respect to different norms.



Whether it is optimal to spatially transport the colors or to change the colors 
depends on the parameters $c_1, c_2, c_3, \alpha$ as well as the norms $\|\cdot\|_u$ and $\|\cdot\|_w$.
With the parameters of Figure~\ref{fig:colorringa} it is optimal to spatially transport the colors,
while with the parameters of Figure~\ref{fig:colorringb} it is optimal to change the colors.

\begin{figure}[h]
\centering
\subfloat[$c_1=1, c_2=1, c_3=1,\alpha =10$.]{\includegraphics[width=0.49\textwidth]{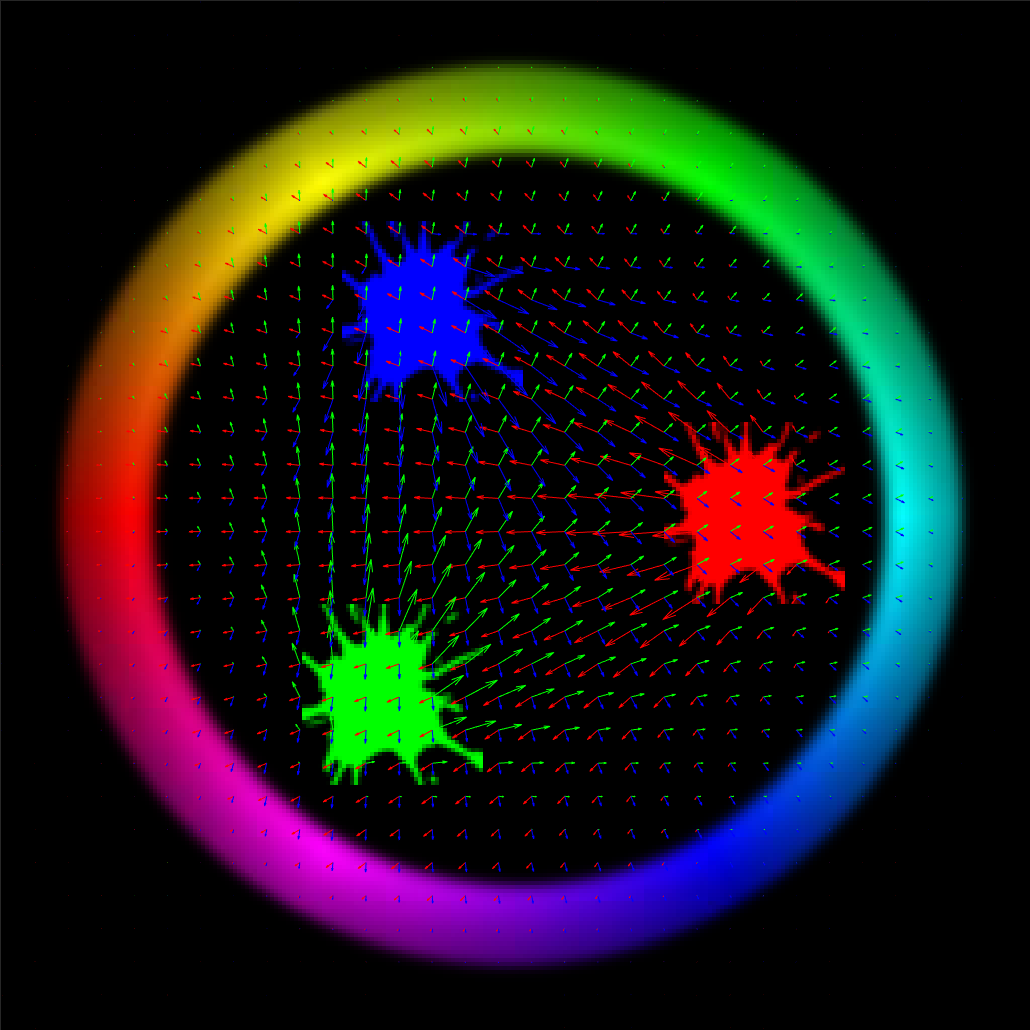}
\label{fig:colorringa}
}
\subfloat[$c_1=1, c_2=1, c_3=1,\alpha =0.1$.]{\includegraphics[width=0.49\textwidth]{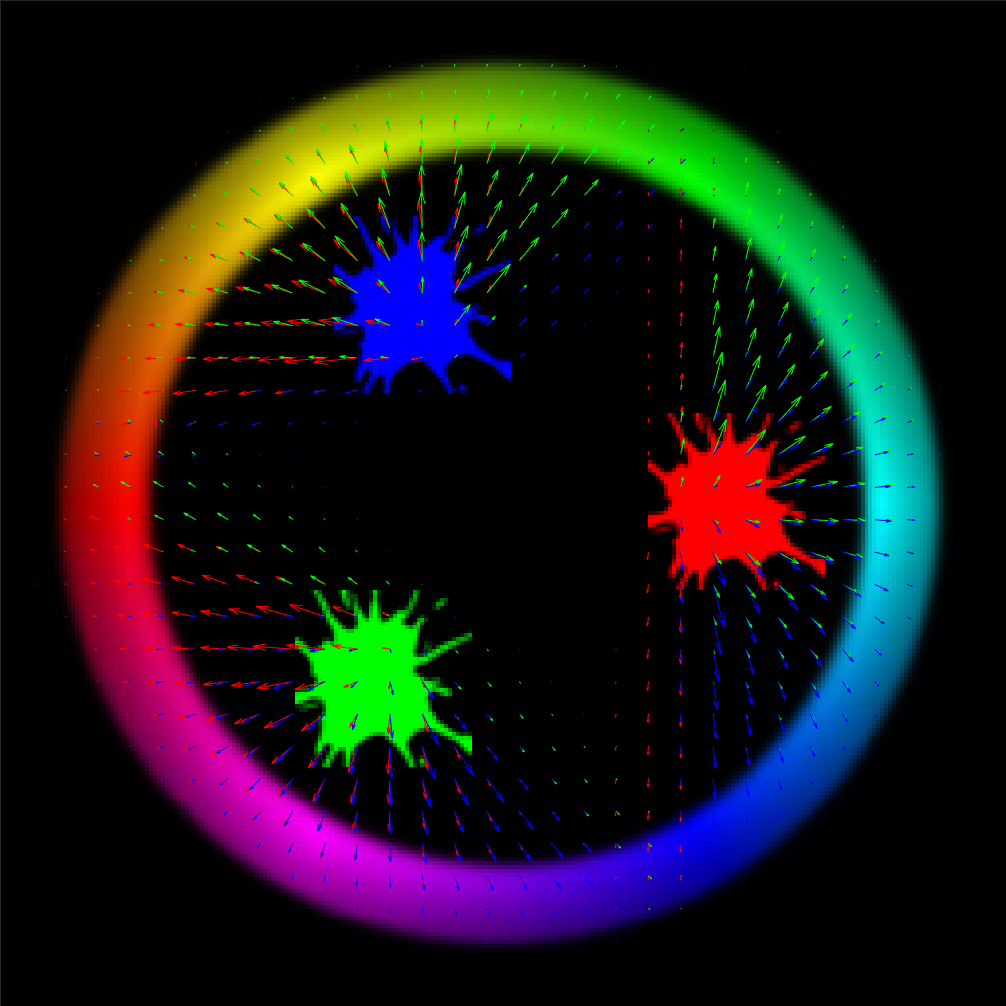}
\label{fig:colorringb}
}
 \caption{Color image vector-OMT example, with a more complicated shape
 and norms $\|\vect{\vu}\|_{1,2}, \|\vect{w}\|_{2}$.
 }
 \label{fig:colorring}
\end{figure}

Finally, we test our algorithm 
on the setup of Figure~\ref{fig:velocityuL2a}
with grid sizes $32\times 32, 64 \times 64, 128 \times 128$, and $256 \times 256$. 
Table~\ref{tab:vecgrid1} shows the number of iterations and runtime tested on a Nvidia Titan Xp GPU
required to achieve a $10^{-3}$ precision, measured as the ratio between the duality gap and primal objective value.
%
\begin{table}
\centering
\begin{tabular}[h]{| c | c | c | c |c| c| c|}
\hline
Grid Size & Iteration count & Time per-iter& Iteration time&$\tau$ & $V(\vect{\lambda}^0,\vect{\lambda}^1)$ \\
\hline
$32 \times 32$ & $0.5\times 10^4$  & $27 \mu s$ &$0.14s$& 1  & 0.57\\
$64 \times 64$ & $0.5\times 10^4$  & $27 \mu s$ &$0.14s$& 1  & 0.57\\
$128 \times 128$ & $2\times 10^4$ & $23 \mu s$ &$0.47s$& 3  & 0.57\\
$256 \times 256$ & $2\times 10^4$ &$68 \mu s$ & $1.36s$&3  & 0.57\\
\hline
\end{tabular}
\caption{Computation cost for vector-OMT as a function of grid size.}
\label{tab:vecgrid1}
\end{table}

\subsection{Diffusion tensor imaging}
Diffusion tensor imaging (DTI) is a technique used in magnetic resonance imaging.
DTI captures orientations and intensities of brain fibers at each spatial position as ellipsoids
and gives us a matrix-valued density.
Therefore, the metric defined by the matrix-OMT 
provides a natural way to compare the differences between brain diffusion images.
In Figure~\ref{fig:dti} we visualize diffusion tensor images
by using colors to indicate different orientations of the tensors at each voxel.
In this paper, we present simple 2D examples as a proof of concept
and leave actual 3D imaging for a topic of future work.
\begin{figure}[h]
\centering
\includegraphics[width=0.49\textwidth]{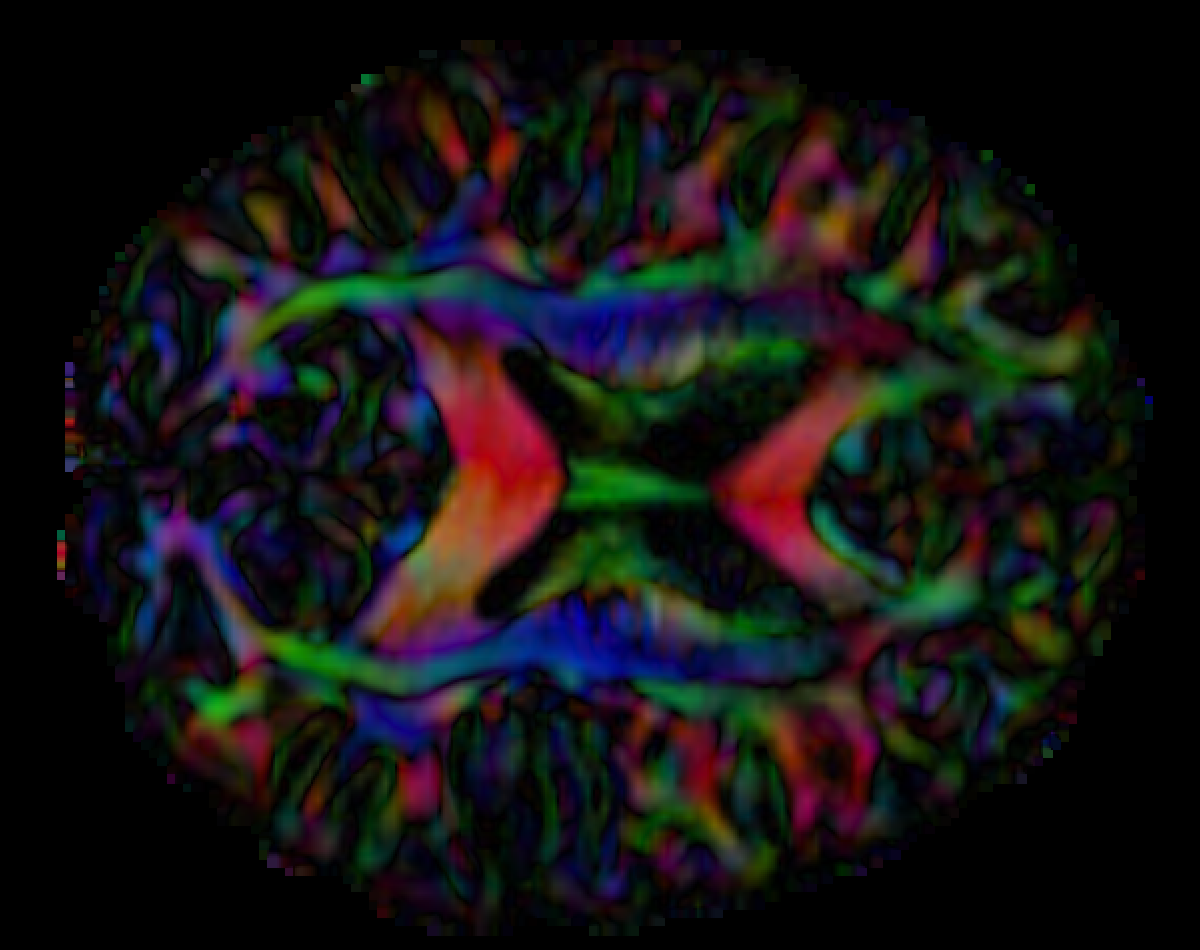}
\includegraphics[width=0.49\textwidth]{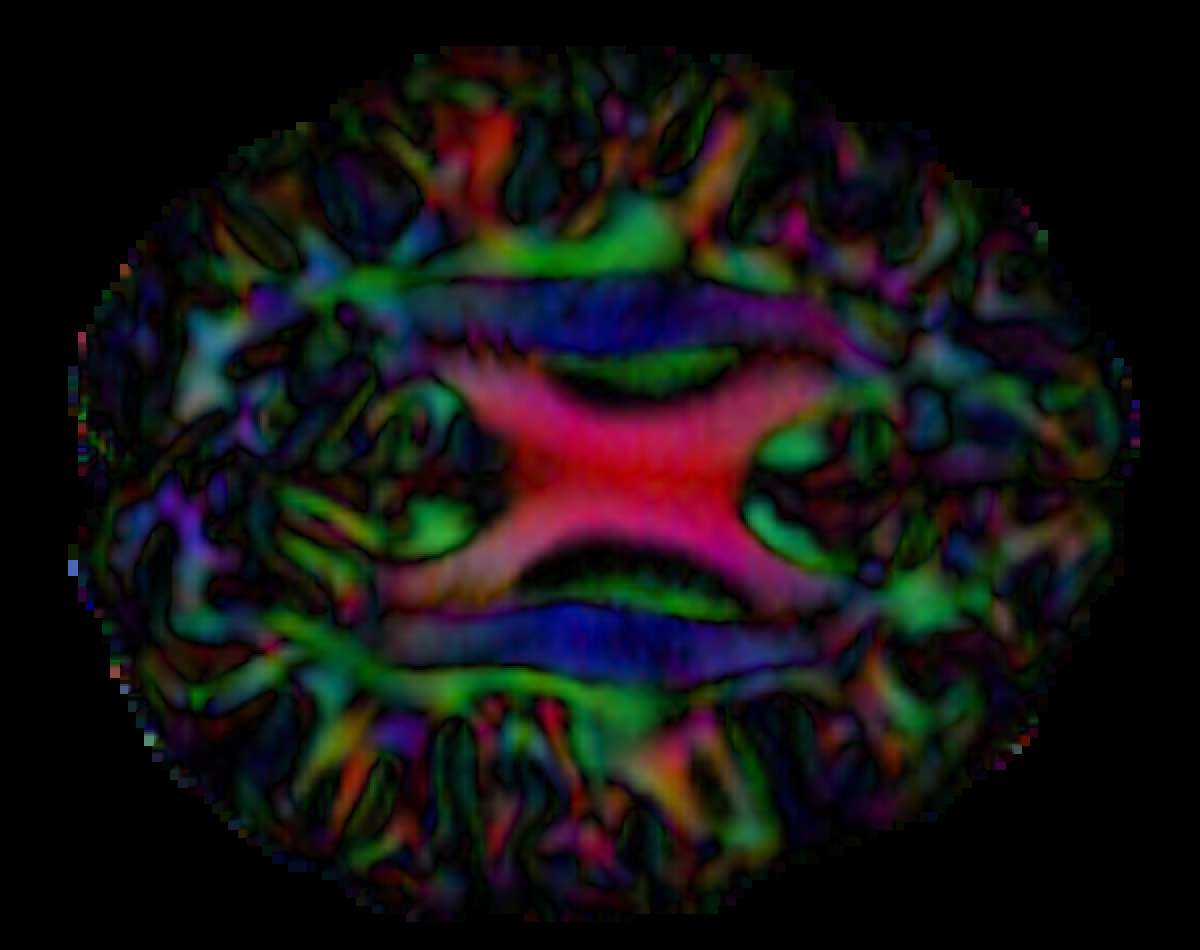}
 \caption{Example of 2D diffusion tensor images.}
 \label{fig:dti}
\end{figure}

Consider three synthetic matrix-valued densities $\Lambda^0$, $\Lambda^1$, and $\Lambda^2$ in Figure~\ref{fig:matrixmarginals}.
The densities $\Lambda^0$ and $\Lambda^1$ have 
mass at the same spatial location,
but the ellipsoids have different shapes.
The densities $\Lambda^0$ and $\Lambda^2$ have the same ellipsoids at different spatial locations.

We compute the distance between $\Lambda^0$, $\Lambda^1$, and $\Lambda^2$
for different parameters $\alpha$ and fixed $\vL=[L_1, L_2]^*$ with 
	\[
		L_1 = \left[\begin{matrix}
		1 & 0 & 0\\
		0 & 2 & 0\\
		0 & 0 & 0
		\end{matrix}\right],
		\quad
		L_2 = \left[\begin{matrix}
		1 & 1 & 1\\
		1 & 0 & 0\\
		1 & 0 & 0
		\end{matrix}\right].		
	\]

Table~\ref{tab:param}, shows the results
with $\|\vU\|_2$  and $\|\vW\|_1$ and grid size $128 \times 128$.
As we can see, whether $\Lambda^0$ is more ``similar'' to $\Lambda^1$ or $\Lambda^2$,
i.e., whether
$M(\lambda^0,\lambda^1)<M(\lambda^0,\lambda^2)$ 
or
$M(\lambda^0,\lambda^1)>M(\lambda^0,\lambda^2)$,
depends on whether the cost on $\vU$, spatial transport, is higher than the cost on $\vW$, 
changing the ellipsoids.

\begin{figure}[h]
\centering
\subfloat[$\Lambda^0$]{\includegraphics[width=0.33\textwidth]{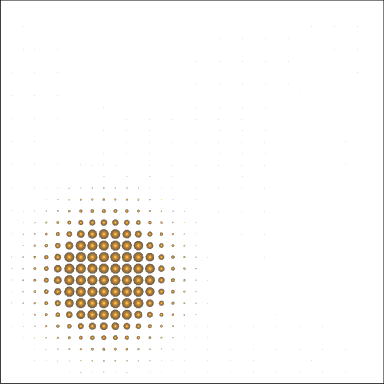}}
\subfloat[$\Lambda^1$]{\includegraphics[width=0.33\textwidth]{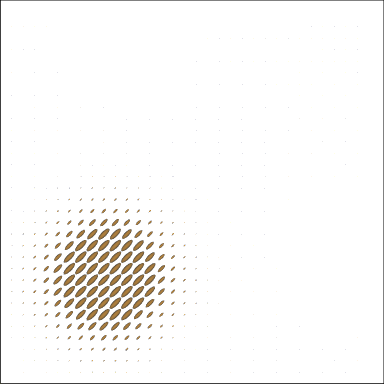}}
\subfloat[$\Lambda^2$]{\includegraphics[width=0.33\textwidth]{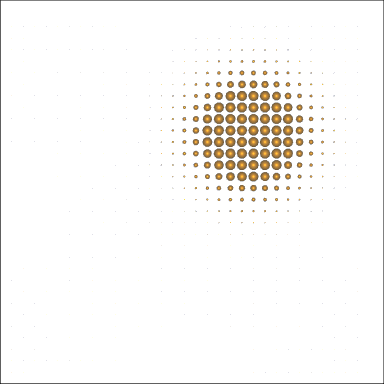}}
 \caption{Synthetic matrix-valued distributions.}
 \label{fig:matrixmarginals}
\end{figure}
\begin{table}
\centering
\begin{tabular}[h]{| c | c | c | c |}
\hline
Parameters & $M(\Lambda^0,\Lambda^1)$ & $M(\Lambda^0,\Lambda^2)$ & $M(\Lambda^1,\Lambda^2)$ \\
\hline
$\alpha=10$ & 2.71 & 0.27 & 3.37 \\
$\alpha=3$ & 0.81  & 0.27  & 1.08 \\
$\alpha=1$ & 0.27 & 0.27 & 0.54 \\
$\alpha=0.3$ & 0.081 & 0.27 & 0.35 \\
$\alpha=0.1$ & 0.027 & 0.27 & 0.29 \\
\hline
\end{tabular}
\caption{Distances between the three images $\Lambda^0,\Lambda^1$ and $\Lambda^2$.}
\label{tab:param}
\end{table}

Again, we test our algorithm on the setup of Figure~\ref{fig:matrixmarginals} 
with $\alpha=1$
on grid sizes $32\times 32, 64 \times 64, 128 \times 128$, and $256 \times 256$. Table~\ref{tab:matrixOMT1}
shows the number of iterations and runtime tested on a Nvidia Titan Xp GPU
required 
to achieve a precision of $10^{-3}$, measured as the ratio between the duality gap and primal objective value.

\begin{table}
\centering
\begin{tabular}[h]{| c | c | c | c | c| c|}
\hline
Grid Size & Iteration count & Time per-iter& Iteration time & $\tau$ & $M(\Lambda^0,\Lambda^1)$ \\
\hline
$32 \times 32$ & $1\times 10^4$ & $39\mu s$ & $0.39s$&10  & 0.27\\
$64 \times 64$ & $1.5\times 10^4$ & $39\mu s$ & $0.59s$&10  & 0.27\\
$128 \times 128$ &  $2\times 10^4$& $85\mu s$ & $1.70s$&30  & 0.27\\
$256 \times 256$ &  $4\times 10^4$&$330\mu s$ & $13.2s$&60  & 0.27\\
\hline
\end{tabular}
\caption{Computation cost for matrix-OMT as a function of grid size.}
\label{tab:matrixOMT1}
\end{table}

\subsection{Discussion}
The best choice of the norms for the vector and matrix-OMT problem
depends on the application.
In color image processing, for example, we recommend to use $\|\cdot\|_{1,2}$ for the $\|\cdot\|_u$ norm (which penalizes transport over space)
since it respects the Euclidean geometry in the spatial domain and the penalties on different colors are decoupled.
On the other hand, the choice for $\|\cdot\|_w$ norm (which penalizes change in color)
seems to matter less.
As another example, the nuclear norm is relevant to the problem of
Michell trusses in structural mechanics \cite{gangbo2008}.

%

The GPU implementation provides significant acceleration.
For reference,
we compared a serial CPU implementation 
against a parallel CUDA implementation
for the scalar-OMT problem.
When we run the method on a
$1024\times1024$ image,
the serial implementation took 
$157.8ms$ per iteration
while the CUDA implementation took
$1.27ms$ per iteration.
For computational problems with very simple parallelizable structures,
$50$-fold to $150$-fold speedup is common.
An Intel Core i7 CPU X990 @ 3.47GHz
and Nvidia Titan Xp GPU were used for this experiment.

\section{Conclusions}
In this paper, we studied the extensions of Wasserstein-1 optimal transport to vector and matrix-valued densities.
This extension, as a tool of applied mathematics, is interesting
if the mathematics is sound, if the numerical methods are good, and if the applications are interesting.
In this paper, we investigated all three concerns.


From a practical viewpoint, that we can solve vector and matrix-OMT problems of realistic sizes in modest time with GPU computing is the most valuable observation.
Applying our algorithms to tackle real world problems in signal/imaging processing, medical imaging, and machine learning would be interesting directions of future research.

Another interesting direction of study is quadratic regularization.
In general, the solutions to vector and matrix-OMT problems are not unique.
However, the regularized version of \eqref{vomt} 
\begin{equation*}
\begin{array}{ll}
\underset{\vect{\vu},\vect{w}}{\text{minimize}}&
\int_{\Omega}\|\vect{\vu}(\vx)\|_u +\alpha\|\vect{w}(\vx)\|_w+ \epsilon (\|\vect{\vu}(\vx)\|_u^2+\|\vect{w}(\vx)\|_w^2)\;d\vx\\
\mbox{subject to} &
\divg_\vx (\vect{\vu})(\vx)
+
\divg_\mathcal{G}(\vect{w}(\vx))=\vect{\lambda}^0(\vx)-\vect{\lambda}^1(\vx)\\
& \vect{\vu} \text{ satisfies zero-flux b.c}
\end{array}
\end{equation*}
is strictly convex and therefore has a unique solution. A similar regularization can be done for matrix-OMT.
As discussed in \cite{LiRyuOsh17, L1partial},
this form of regularization is particularly useful as a slight modification to the 
proposed method solves the regularized problem.

%
%
%
%
%
%

\section*{Acknowledgments}
We would like to thank Wilfrid Gangbo for many fruitful and inspirational discussions on the related topics.
The Titan Xp?s used for this research were donated by the NVIDIA Corporation. This work was funded by
ONR grants N000141410683, N000141210838, DOE grant de-sc00183838 and a startup funding from Iowa
State University.

\appendix

\section{Discretization}
Here, we describe the discretization for the vector-OMT problem
\eqref{vomt} and \eqref{vomt-dual}.
The discretization for the matrix-OMT problem is similar.
We consider the 2D square domain, but (with more complicated notation)
our approach
immediately generalizes to more general domains.
Again, we use the same symbol for the discretization and its continuous counterpart.

Consider a $n\times n$ discretization of $\Omega$ with finite difference $\Delta x$ in both $x$ and $y$ directions.
Write the $x$ and $y$ coordinates of the points as $x_1,\dots,x_n$ and $y_1,\dots,y_n$.
So we are approximating the domain $\Omega$ with $\{x_1,\dots,x_n\}\times\{y_1,\dots,y_n\}$.
Write $C(x,y)$ for the $\Delta x\times\Delta x$ cube centered at $(x,y)$, i.e., 
\[
C(x,y)=\{
(x',y')\in \mathbb{R}^2\,|\,
|x'-x|\le \Delta x/2
\,,\,
|y'-y|\le \Delta x/2
\}\ .
\]

We use a finite volume approximation for $\vect{\lambda}^0$, $\vect{\lambda}^1$, $\vect{w}$ and $\vect{\phi}$.
Specifically, we write $\vect{\lambda}^0\in \mathbb{R}^{n\times n\times k}$ with
\[
\vect{\lambda}^0_{ij}\approx \int _{C(x_i,y_j)}\vect{\lambda}^0(x,y)\;dxdy,
\]
for $i,j=1,\dots,n$. The discretizations $\vect{\lambda}^1,\vect{\phi}\in \mathbb{R}^{n\times n\times k}$ 
and $\vect{w}\in \mathbb{R}^{n\times n \times \ell}$ are defined the same way.

Write $\vect{\vu}=(\vect{u}_x,\vect{u}_y)$
for both the continuous variables and their discretizations.
To be clear, the subscripts of $\vect{u}_x$ and $\vect{u}_y$ do not denote differentiation.
We use the discretization $\vect{u}_x\in \mathbb{R}^{(n-1)\times n\times k}$
and $\vect{u}_y\in \mathbb{R}^{n\times (n-1)\times k}$.
For $i=1,\dots,n-1$ and $j=1,\dots,n$
\[
\vect{u}_{x,ij} \approx \int_{C(x_i+\Delta x/2,y_j)}\vect{u}_x(x,y)\;dxdy,
\]
and for $i=1,\dots,n$ and $j=1,\dots,n-1$
\[
\vect{u}_{y,ij} \approx \int_{C(x_i,y_j+\Delta x/2)}\vect{u}_y(x,y)\;dxdy.
\]
In defining $\vect{u}_x$ and $\vect{u}_y$, the center points are placed
between the $n\times n$ grid points to make the finite difference operator symmetric.

Define the discrete spacial divergence operator
$\divg_\vx (\vect{\vu})\in \mathbb{R}^{n\times n\times k}$
as
\[
\divg_\vx (\vect{\vu})_{ij}
=\frac{1}{\Delta x}
(\vect{u}_{x,ij}-\vect{u}_{x,(i-1)j}+\vect{u}_{y,ij}-\vect{u}_{y,i(j-1)}),
\]
for $i,j=1,\dots,n$,
where we mean $\vect{u}_{x,0j}=\vect{u}_{x,nj}=0$
for $j=1,\dots,n$ and $\vect{u}_{y,i0}=\vect{u}_{y,in}=0$
for $i=1,\dots,n$.
This definition of $\divg_\vx (\vect{\vu})$
makes the discrete approximation consistent with the zero-flux boundary condition.

For $\vect{\phi}\in \mathbb{R}^{n\times n}$,
define the discrete gradient operator
$\nabla_\vx \vect{\phi}
=((\nabla \vect{\phi})_x,(\nabla \vect{\phi})_y)$
as 
\begin{align*}
(\nabla \vect{\phi})_{x,ij}=
(1/\Delta x)
\left(\vect{\phi}_{i+1,j}-\vect{\phi}_{i,j}\right)&\quad\text{for }
i=1,\dots,n-1,\,j=1,\dots,n\\
(\nabla \vect{\phi})_{y,ij}=
(1/\Delta x)
\left(\vect{\phi}_{i,j+1}-\vect{\phi}_{i,j}\right)&\quad\text{for }
i=1,\dots,n,\,j=1,\dots,n-1.
\end{align*}
So 
$(\nabla\vect{\phi})_x\in \mathbb{R}^{ (n-1)\times n\times  k}$ and $(\nabla\vect{\phi})_y\in \mathbb{R}^{ n\times (n-1)\times  k}$,
and the  $\nabla_\vx$ is the transpose (as a matrix) of $-\divg_\vx$.

Ghost cells are convenient
for both describing and implementing the method.
This approach is similar to that of 
\cite{LiRyuOsh17}.
We redefine the variable
$\vect{\vu}=(\vect{u}_x,\vect{u}_y)$
so that
\begin{align*}
\vect{u}_{x,ij}=
\left\{
\begin{array}{ll}
\vect{u}_{x,ij}
&\text{for }
i<n\\
0&\text{for }i=n
\end{array}
\right.
\qquad
\vect{u}_{y,ij}=
\left\{
\begin{array}{ll}
\vect{u}_{y,ij}&\text{for }
j<n\\
0&\text{for }j=n\ ,
\end{array}
\right.
\end{align*}
for $i,j=1,\dots,n$, and $\vect{u}_x,\vect{u}_y\in \mathbb{R}^{ n\times n\times k}$.
We also redefine  $\nabla\vect{\phi}=((\nabla \vect{\phi})_x,(\nabla \vect{\phi})_y)$
so that 
\begin{align*}
(\nabla \vect{\phi})_{x,ij}=
\left\{
\begin{array}{ll}
(\nabla \vect{\phi})_{x,ij}
&\text{for }
i<n\\
0&\text{for }i=n
\end{array}
\right.\qquad
(\nabla \vect{\phi})_{y,ij}=
\left\{
\begin{array}{ll}
(\nabla \vect{\phi})_{y,ij}&\text{for }
j<n\\
0&\text{for }j=n\ ,
\end{array}
\right.
\end{align*}
for $i,j=1,\dots,n$, and
$(\nabla \vect{\phi})_{x},(\nabla \vect{\phi})_{y}\in \mathbb{R}^{n\times n\times k}$.

With some abuse of notation, we write
\[
\|\vect{\vu}\|_u=\sum^n_{i=1}\sum^n_{j=1}
\|(\vect{u}_{x,ij},\vect{u}_{y,ij})\|_u
\qquad
\|\vect{w}\|_w=\sum^n_{i=1}\sum^n_{j=1}
\|\vect{w}_{ij}\|_w.
\]
%
%
%
Using this notation, we write the discretization of \eqref{vomt} as
\begin{equation*}
\begin{aligned}
& \underset{\vect{\vu},\vect{w}}{\text{minimize}}
& &   \|\vect{\vu}\|_{u}  + \alpha \|\vect{w}\|_{w}  \\
& \text{subject to}
& & \divg_\vx (\vect{\vu})+\divg_\cG(\vect{w})=\vect{\lambda}^0-\vect{\lambda}^1,
\end{aligned}
\end{equation*}
where the boundary conditions are implicitly handled by the discretization.


\section{Shrink operators with closed-form solutions}
Define $\srk_1:\mathbb{C}\rightarrow \mathbb{C}$ as
\begin{align*}
\srk_1(x;\mu)
&=
\argmin_{z\in \mathbb{C}}\left\{
\mu |z|+(1/2)|z-x|^2
\right\}\\
&=
\left\{
\begin{array}{ll}
(1-\mu/|x|)x&\text{ for }|x|\ge \mu\\
0&\text{otherwise}
\end{array}
\right.
\end{align*}
for $\mu> 0$.
Define $\srk_2:\mathbb{C}^k\rightarrow \mathbb{C}^k$ as
\begin{align*}
\srk_2(x;\mu)
&=\argmin_{z\in \mathbb{C}^k}\left\{
\mu \|z\|_2+(1/2)\|z-x\|^2_2
\right\}\\
&=\left\{
\begin{array}{ll}
(1-\mu/\|x\|_2)x&\text{ for }\|x\|_2\ge \mu\\
0&\text{otherwise}
\end{array}
\right.
\end{align*}
for $\mu>0$, where $\|\cdot\|_2$ is the standard Euclidean norm.
Define $\srk_\mathrm{nuc}:\cC\rightarrow \cC$ as
\begin{align*}
\srk_\mathrm{nuc}(X;\mu)
&=\argmin_{Z\in \cC}\left\{
\mu \|Z\|_*+(1/2)\|Z-X\|^2_2
\right\}\\
&=
U\diag\left( [(\sigma_1-\mu)_+,\dots,(\sigma_n-\mu)_+]\right)
V^T
\end{align*}
where $\|\cdot\|_*$ is the nuclear norm and 
$X=U\diag\left( [\sigma_1,\dots,\sigma_n]\right)V^T$
is the singular-value decomposition of $X$
\cite{CaiCanShe10}.
If $x=(x_1,x_2,\dots,x_k)$ and 
$\|x\|=\|x_1\|_1+\|x_2\|_2+\dots+ \|x_k\|_k$ for some norms
$\|\cdot\|_1,\dots,\|\cdot\|_k$,
then the shrink operators can be applied component-by-component.
So if $\srk_1,\srk_2,\dots,\srk_k$ are the individual shrink operators, 
then
\[
\srk(x;\mu)=\begin{bmatrix}
\srk_1(x_1;\mu)\\
\srk_2(x_2;\mu)\\
\vdots\\
\srk_k(x_k;\mu)\\
\end{bmatrix}.
\]
All shrink operators we consider in this paper can be built from these shrink operators.
These ideas are well-known to the compressed sensing and proximal methods community
\cite{parikh2014}.

However, there is a subtlety we must address 
when applying the shrink operators to matrix-OMT:
the $\vU$ update is defined as the minimization over $\cH$, not all of $\cC$,
and the $\vW$ update is defined as the minimization over $\cS$, not all of $\cC$.
Fortunately, this is not a problem
when we use unitarily invariant norms (such as the nuclear norm) thanks to the following lemma.

\begin{lemma}
\label{lm:unitary}
Let $M\in \cH$ and $N\in \cS$.
Let $\|\cdot\|$ be a unitarily invariant norm, i.e.,
$\|UAV\|=\|A\|$ for any $A\in \cC$, and $U,V\in \cC$ unitary.
Then
\[
M^+=
\argmin_{A\in \cH}
\left\{
\|A\|+\frac{1}{2}
\|A-M\|^2_2
\right\}=
\argmin_{A\in \cC}
\left\{
\|A\|+\frac{1}{2}
\|A-M\|^2_2
\right\}
\]
is Hermitian
\[
N^+=
\argmin_{A\in \cS}
\left\{
\|A\|+\frac{1}{2}
\|A-N\|^2_2
\right\}=
\argmin_{A\in \cC}
\left\{
\|A\|+\frac{1}{2}
\|A-N\|^2_2
\right\}
\]
is Skew-Hermitian.
\end{lemma}

\begin{proof}
The minimum over $\cH\subset \cC$ is the same as the minimum over all of $\cC$,
if the minimum over all of $\cC$ is in $\cH$.
Likewise, 
the minimum over $\cS\subset \cC$ is the same as the minimum over all of $\cC$,
if the minimum over all of $\cC$ is in $\cS$.

Write $\sigma:\cC\rightarrow \mathbb{R}^n$
for the function that outputs singular values
in decreasing order.
Because of unitary invariance, we can write
$\|X\|=f(\sigma(X))$,
where
$f:\mathbb{R}^n\rightarrow \mathbb{R}$ is a norm on $\mathbb{R}^n$
\cite{Neumann1937}.
Furthermore, it's subdifferential can be written as
\cite{Lewis1995}
\[
\partial \|\cdot\|(X)=
\{U
\diag(\mu)
V^T\,|\,
\mu\in(\partial f) (\sigma(X)),\,
U\diag\sigma(X)V^T=X\text{ is SVD}
\}.
\]

Write $M=U\Sigma U^T$
for $M$'s eigenvalue decomposition, which is also
its singular value decomposition.
Define
\[
\sigma^+=\argmin_{s}\{
f(s)+(1/2)\|s-\sigma(M)\|_2^2
\},
\]
which implies
$0\in \partial f(\sigma^+)+\sigma^+-\sigma(M)$.
With this, we can verify that 
\[
0\in U\diag(\partial f(\sigma^+))U^T+U\diag(\sigma^+)U^T
-U\diag(\sigma(M)) U^T,
\]
i.e., $U\diag(\sigma^+)U^T$ satisfies the optimality
conditions of the optimization problem that defines $M^+$.
So 
$M^+=U\diag(\sigma^+)U^T$, which is Hermitian.

Likewise, write $N=U\Lambda U^T$ for
$N$'s eigenvalue decomposition.
$N$ has orthonormal eigenvectors
and its eigenvalues are purely imaginary.
We separate out the magnitude and phase of $\Lambda$
and write
$N=U\diag(\sigma(N))PU^T$.
More precisely, $\sigma(N)= |\diag(\Lambda)|$
and $P$ is a diagonal matrix with diagonal components $\pm i$. Then the SVD of $N$ is
$N=U\diag(\sigma(N))V^T$
where $PU^T=V^T$.
With the same argument as before, we conclude that
$N^+=U\diag(\sigma^+)V^T$
for some $\sigma^+\in \mathbb{R}^n$.
So
$N^+=U\Lambda^+U^T$
where $\Lambda^+=\diag(\sigma^+)P$ is diagonal and purely imaginary,
and we conclude $N^+$ is skew-Hermitian.
\end{proof}

The norm $\|\cdot\|_1$ as described in Section~\ref{ss:shrink}, 
is not unitarily invariant.
However, its shrink operator acts element-by-element,
so it is easy to arrive at a conclusion similar to that of Lemma~\ref{lm:unitary}.

\bibliographystyle{siamplain}
\bibliography{./refs}

\bibliographystyle{siamplain}
\bibliography{./refs}

\end{document}